\documentclass[reqno,a4paper]{amsart}

\usepackage[T1]{fontenc}
\usepackage[utf8]{inputenc}
\usepackage{lmodern}
\usepackage{microtype}

\usepackage{comment}

\usepackage{fullpage}

\usepackage{subcaption}

\usepackage{supertabular}
\usepackage{booktabs}
\usepackage{longtable}
\usepackage[margin=0pt]{subcaption}

\usepackage{tikz}
\usetikzlibrary{knots,cd,backgrounds}
\usetikzlibrary{math}
\tikzset{
  x=0.4cm,y=0.4cm,
  ext-vert/.append style={circle,draw=black,fill=black,inner sep=0pt,minimum size=0.8mm},
  inn-vert/.append style={circle,draw=black,fill=black,inner sep=0pt,minimum size=0.8mm},
  bgrect/.append style={draw=none}
}

\usepackage{color}

\usepackage{amssymb} 

\usepackage{hyperref} 
\usepackage[hyperpageref]{backref}
\usepackage[nobysame,alphabetic,initials,msc-links]{amsrefs}

\DefineSimpleKey{bib}{how}
\renewcommand{\PrintDOI}[1]{%
  \href{http://dx.doi.org/#1}{{\tt DOI:#1}}%
}
\renewcommand{\eprint}[1]{#1}
\BibSpec{book}{%
    +{}  {\PrintPrimary}                {transition}
    +{.} { \PrintDate}                  {date}
    +{.} { \textit}                     {title}
    +{.} { }                            {part}
    +{:} { \textit}                     {subtitle}
    +{,} { \PrintEdition}               {edition}
    +{}  { \PrintEditorsB}              {editor}
    +{,} { \PrintTranslatorsC}          {translator}
    +{,} { \PrintContributions}         {contribution}
    +{,} { }                            {series}
    +{,} { \voltext}                    {volume}
    +{,} { }                            {publisher}
    +{,} { }                            {organization}
    +{,} { }                            {address}
    +{,} { }                            {status}
    +{,} { \PrintDOI}                   {doi}
    +{,} { \PrintISBNs}                 {isbn}
    +{}  { \parenthesize}               {language}
    +{}  { \PrintTranslation}           {translation}
    +{;} { \PrintReprint}               {reprint}
    +{.} { }                            {note}
    +{.} {}                             {transition}
    +{}  {\SentenceSpace \PrintReviews} {review}
}
\BibSpec{article}{%
    +{}  {\PrintAuthors}                {author}
    +{,} { \textit}                     {title}
    +{.} { }                            {part}
    +{:} { \textit}                     {subtitle}
    +{,} { \PrintContributions}         {contribution}
    +{.} { \PrintPartials}              {partial}
    +{,} { }                            {journal}
    +{}  { \textbf}                     {volume}
    +{}  { \PrintDatePV}                {date}
    +{,} { \issuetext}                  {number}
    +{,} { \eprintpages}                {pages}
    +{,} { }                            {status}
    +{,} { \PrintDOI}                   {doi}
    +{,} { \eprint}        {eprint}
    +{}  { \parenthesize}               {language}
    +{}  { \PrintTranslation}           {translation}
    +{;} { \PrintReprint}               {reprint}
    +{.} { }                            {note}
    +{.} {}                             {transition}
    +{}  {\SentenceSpace \PrintReviews} {review}
}
\BibSpec{collection.article}{%
    +{}  {\PrintAuthors}                {author}
    +{,} { \textit}                     {title}
    +{.} { }                            {part}
    +{:} { \textit}                     {subtitle}
    +{,} { \PrintContributions}         {contribution}
    +{,} { \PrintConference}            {conference}
    +{}  {\PrintBook}                   {book}
    +{,} { }                            {booktitle}
    +{,} { \PrintDateB}                 {date}
    +{,} { pp.~}                        {pages}
    +{,} { }                            {publisher}
    +{,} { }                            {organization}
    +{,} { }                            {address}
    +{,} { }                            {status}
    +{,} { \PrintDOI}                   {doi}
    +{,} { \eprint}        {eprint}
    +{}  { \parenthesize}               {language}
    +{}  { \PrintTranslation}           {translation}
    +{;} { \PrintReprint}               {reprint}
    +{.} { }                            {note}
    +{.} {}                             {transition}
    +{}  {\SentenceSpace \PrintReviews} {review}
}
\BibSpec{misc}{%
  +{}{\PrintAuthors}  {author}
  +{,}{ \textit}      {title}
  +{.}{ }             {how}
  +{}{ \parenthesize} {date}
  +{,} { available at \eprint}        {eprint}
  +{,}{ available at \url}{url}
  +{,}{ }             {note}
  +{.}{}              {transition}
}

\numberwithin{equation}{section}
\counterwithin{figure}{section}
\counterwithin{table}{section}

\newtheorem{theorem}{Theorem}[section]
\newtheorem{corollary}[theorem]{Corollary}
\newtheorem{lemma}[theorem]{Lemma}
\newtheorem{proposition}[theorem]{Proposition}
\newtheorem{conjecture}[theorem]{Conjecture}
\newtheorem{step}{Step}
\makeatletter
\@addtoreset{step}{theorem}
\makeatother

\theoremstyle{remark}
\newtheorem{remark}[theorem]{Remark}

\newtheorem{example}[theorem]{Example}

\theoremstyle{definition}

\mathchardef\mhyph="2D        
 
\newcommand{\Z}{\mathbb{Z}}

\newcommand{\bT}{\mathbb{T}}
\newcommand{\bC}{\mathbb{C}}
\newcommand{\bF}{\mathbb{F}}

\newcommand{\bP}{\mathbb{P}}
\newcommand{\cC}{\mathcal{C}}

\newcommand{\cO}{\mathcal{O}}

\newcommand{\rH}{\mathrm{H}}
\newcommand{\rK}{\mathrm{K}}
\newcommand{\rL}{\mathrm{L}}
\newcommand{\rO}{\mathrm{O}}
\newcommand{\rS}{\mathrm{S}}
\newcommand{\rV}{\mathrm{V}}

\newcommand{\SO}{\mathrm{SO}}
\newcommand{\PO}{\mathrm{PO}}
\newcommand{\SU}{\mathrm{SU}}
\newcommand{\Spin}{\mathrm{Spin}}
\newcommand{\Pin}{\mathrm{Pin}}
\newcommand{\Sp}{\mathrm{Sp}}
\newcommand{\PSp}{\mathrm{PSp}}
\newcommand{\VO}{\mathrm{VO}}
\newcommand{\PG}{\mathrm{PG}}
\newcommand{\id}{\mathrm{id}}
\newcommand{\Hilbf}{\mathrm{Hilb}_\mathrm{f}}
\newcommand{\medwedge}{{\textstyle\bigwedge}}

\DeclareMathOperator{\End}{End}
\DeclareMathOperator{\Hom}{Hom}

\DeclareMathOperator{\tr}{tr}

\DeclareMathOperator{\Rep}{Rep}
\DeclareMathOperator{\Aut}{Aut}
\DeclareMathOperator{\Modcat}{Mod}
\DeclareMathOperator{\Sym}{Sym}

\newcommand{\absv}[1]{\left|#1\right|}
\newcommand{\ang}[1]{\left\langle#1\right\rangle}
\newcommand{\rmodc}[1]{\mathrm{mod}\mhyph#1}
\newcommand{\lmodc}[1]{#1\mhyph\mathrm{mod}}
\newcommand{\bmodc}[1]{#1\mhyph\mathrm{mod}\mhyph#1}

\begin{document}
\date{July 28, 2025}
\author{Simon Schmidt}
\address{Simon Schmidt, Ruhr University Bochum}
\email{s.schmidt@rub.de}
\author{Makoto Yamashita}
\address{Makoto Yamashita, University of Oslo}
\email{makotoy@math.uio.no}
\thanks{M.Y.: this research was funded, in part, by The Research Council of Norway [projects 300837 and 324944].
S. Sch has received funding from the European Union’s Horizon 2020 research and innovation programme under the Marie Sklodowska-Curie grant agreement No. 101030346. He acknowledges support by
the Deutsche Forschungsgemeinschaft (DFG, German Research Foundation) under Germany’s Excellence Strategy - EXC 2092 CASA - 39078197}

\title{Quantum symmetry of $3$-transitive graphs}

\begin{abstract}
We study the quantum automorphism group of $3$-transitive graphs in this article. Those are highly symmetric graphs that were classified by Cameron and Macpherson in 1985, and we compute the quantum automorphism group of all such graphs, excluding the orthogonal graphs $\rO^-(6,q)$ for $q>3$.
We show that there is no quantum symmetry for the McLaughlin graph and the orthogonal graphs $\rO^-(6,q)$ with $q = 2, 3$, while that the quantum automorphism group of the affine polar graphs $\VO^{+}(2k,2)$ and $\VO^{-}(2k,2)$ are monoidally equivalent to $\PO(n)$ and $\PSp(n)$, respectively.
We use planar algebras to obtain our results, where the $3$-transitivity of the graphs gives bounds on the dimensions of the $2$-- and $3$-box spaces of the associated planar algebras. 
\end{abstract}

\maketitle

\section{Introduction}
Quantum automorphism groups of finite graphs were introduced by Banica~\cite{MR2146039} and Bichon~\cite{MR1937403}. They are generalizations of the classical automorphism group of a graph in the framework of Woronowicz's compact matrix quantum groups~\cite{MR901157}. Quantum automorphism groups are examples of quantum subgroups of the quantum permutation group $S_n^+$ which was introduced by Wang~\cite{MR1637425}. It is a natural question to compute the quantum automorphism group of certain graphs. In~\cite{MR2335703}, Banica and Bichon determined those for small vertex transitive graphs. More recently, the first author determined the quantum automorphism groups of several families of graphs~\cite{schmidt-thesis}. An important incredient for computing those quantum groups were classical symmetries of the graphs involved. In this article, the classical symmetries also play an important role which is why we restrict to $3$-transitive graphs. 

A graph is called $3$-transitive (or $3$-homogeneous) if any isomorphism between subgraphs of size $3$ is the restriction of an automorphism of the graph. We will make use of the symmetries of the graphs for the computation of their quantum automorphism groups. Cameron and Macpherson classified those graphs in~\cite{MR805453} by classifying rank $3$ permutation groups with rank $3$ subconstituents. For example, the Schläfli graph and the McLaughlin graph are $3$-transitive. 

Planar algebras were first introduced by Jones~\cite{arXiv:math/9909027}. Already in~\cite{MR2146039}, Banica associated a planar subalgebra of the spin planar algebra to the quantum automorphism groups a graph. We focus on $3$-transitive graphs as the $3$-transitivity yields that the planar algebra associated to the graph is singly generated and also bounds the dimension of the $3$-box space by $15$. 

Liu defined and classified singly generated Yang-Baxter planar algebras in~\cite{arXiv:1507.06030}. For those planar algebras the dimension of $3$-box space is also bounded by $15$. Edge~\cite{arXiv:1902.08984} showed which $3$-point regular graphs yield Yang-Baxter planar algebras. A graph is $3$-point regular if the number of common neighbors of three vertices in the graph only depends on the adjacencies between those vertices. Note that $3$-transitive graphs are especially $3$-point regular. Edge furthermore asked for the planar algebras associated to the Schläfli graph and the McLaughlin graph as the relations in the planar algebra are close to those of a Yang-Baxter planar algebra. We will show that both of those graphs have no quantum symmetry, i.e., their quantum automorphism group coincides with their classical automorphism group. 

In our main theorem (Theorem~\ref{thm:main-thm}) we determine the quantum symmetries of all $3$-transitive graphs, except for the graphs $\rO^-(6, q)$ with $q>3$. In the case of graphs which have quantum symmetry, we furthermore obtain monoidal equivalences of the quantum automorphism groups and other known quantum groups. More precisely, we show that the graphs graphs $\rO^-(6, q)$ for $q=2,3$ and the McLaughlin graph do not have quantum symmetry. For the quantum automorphism group of the affine polar graphs $\VO^{+}(2k,2)$ and $\VO^{-}(2k,2)$, we prove that they are monoidally equivalent to $\PO(n)$ and $\PSp(n)$, respectively. The other cases of $3$-transitive graphs were already settled in the literature~\cite{MR1403861},~\cite{MR1469634},~\cite{MR2096666},~\cite{arXiv:2106.08787}. We conjecture that the graphs $\rO^-(6, q)$ with $q>3$ do not have quantum symmetry. 

\paragraph{Acknowledgements}
We thank Alexander Mang and Sergey Neshveyev for fruitful discussions.

\section{Preliminaries}

\subsection{Hopf-algebraic presentation}

Let $X$ be a finite simple graph (unoriented graph without parallel edges and loops) on $N$ vertices.
We denote its adjacency matrix by $A = A_X \in M_N(\{0, 1\})$.

Let $\cO(\rS_N^+)$ be the Hopf $*$-algebra of the free permutation group $\rS_N^+$: it is a universal algebra generated by projections $u_{i j}$ for $0 \le i, j < N$ such that the matrix $u = (u_{i j})_{i, j} \in M_N(\cO(\rS_N^+))$ is a unitary element and projections in the same row or column are orthogonal.
The quantum automorphism group of $X$, which we denote by $\Aut^+(X)$, is the compact quantum group for which the associated Hopf $*$-algebra $\cO(\Aut^+(X))$ is the quotient of $\cO(\rS_N^+)$ with an extra relation $A u = u A$~\cite{MR2335703}.

Let $E = E_X$ be the set of edges of $X$, identified with a symmetric subset of the Cartesian square of the set of vertices.
The relations satisfied by the generators $u_{i j} \in \cO(\Aut^+(X))$ can be equivalently presented as
\begin{align}
&u_{ij} =u_{ij}^*,\, u_{ij}u_{ik}=\delta_{jk}u_{ij},\, u_{ji}u_{ki}=\delta_{jk}u_{ji}, \label{rel1}\\
&\sum_{k}u_{ik} =\sum_{k}u_{ki}=1,\label{rel2}\\
&u_{ij}u_{kl} =0 \text{ if } (i,k)\in E \text{ and } (j,l)\notin E \text{ or vice versa.}\label{rel3}
\end{align}

Let $Q$ denote the algebra of complex functions on the finite set $\{i \in \Z \mid 0 \le i < N \}$.
Together with the state $\phi \colon Q \to \bC$ corresponding to the uniform probability measure, $Q$ represents a C$^*$-Frobenius algebra object in $\Rep \Aut^+(X)$.
As $\Aut^+(X)$ is a quantum group of Kac type, the categorical dimension of $Q$ is the same as the dimension of underlying space, i.e., $\dim_q(Q) = N$.
In particular, $(Q, \phi)$ is a Q-system.

\subsection{Categorical presentation}\label{sec:cat-pres}

Let us present the corresponding structures in the theory of planar algebras~\cite{arXiv:math/9909027}, following the convention of~\cite{MR3922286}.
Given a positive integer $N$, we start with the \emph{spin planar algebra} $P$, which is a shaded planar algebra with $\dim P_{k, \pm} = N^k$ for $k > 0$, $\dim P_{0,-} = N$, and $\dim P_{0,+} = 1$, whose basis are given by the diagrammatic elements as in Figure~\ref{fig:basis-spin-pl-alg}, with labeles $0 \le x_i < N$ and $0 \le x < N$.
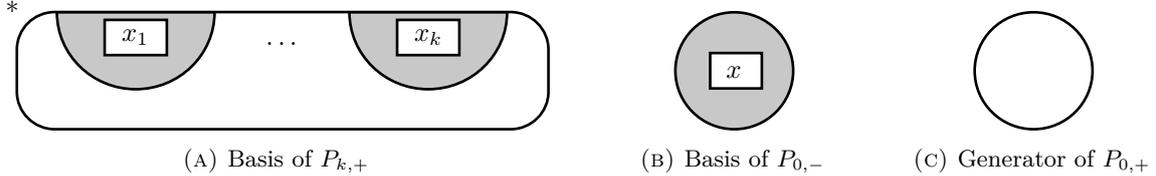
\begin{figure}[h]
\begin{subfigure}{.5\linewidth}
\centering

\definecolor{cc8c8c8}{RGB}{200,200,200}

\def \globalscale {1.000000}
\begin{tikzpicture}[y=1pt, x=1pt, yscale=\globalscale,xscale=\globalscale, every node/.append style={scale=\globalscale}, inner sep=0pt, outer sep=0pt]
  \path[draw=black,line cap=,line width=1.0pt] (14.674, 69.598) -- (185.326, 
  69.598){[rotate=-45.0] arc(225.0:315.0:14.173 and -14.173)} -- (199.499, 
  39.575){[rotate=-135.0] arc(225.0:315.0:14.173 and -14.173)} -- (14.674, 
  25.402){[rotate=-45.0] arc(45.0:135.0:14.173 and -14.173)} -- (0.5, 
  55.425){[rotate=-135.0] arc(45.0:135.0:14.173 and -14.173)} -- cycle;

  \path[draw=black,fill=cc8c8c8,line cap=,line width=1.0pt] (74.512, 69.512).. 
  controls (74.512, 53.482) and (61.299, 40.488) .. (45.0, 40.488).. controls 
  (28.701, 40.488) and (15.488, 53.482) .. (15.488, 69.512).. controls (35.163, 
  69.512) and (54.837, 69.512) .. (74.512, 69.512) -- cycle;

  \path[draw=black,fill=cc8c8c8,line cap=,line width=1.0pt] (184.012, 69.512).. 
  controls (184.012, 53.482) and (170.799, 40.488) .. (154.5, 40.488).. controls
   (138.201, 40.488) and (124.988, 53.482) .. (124.988, 69.512).. controls 
  (144.663, 69.512) and (164.337, 69.512) .. (184.012, 69.512) -- cycle;

  \path[draw=black,fill=white,line cap=,line width=1.0pt] (33.359, 66.641) 
  rectangle (56.641, 53.359);

  \path[draw=black,fill=white,line cap=,line width=1.0pt] (142.859, 66.641) 
  rectangle (166.141, 53.359);

  \node[line cap=,line width=1.0pt,anchor=south west] (text3) at (39.79, 
  57.153){$x_1$};

  \node[line cap=,line width=1.0pt,anchor=south west] (text4) at (149.29, 
  57.153){$x_k$};

  \node[line cap=,line width=1.0pt,anchor=south west] (text5) at (93.79, 
  57.153){$\dots$};

  \node[line cap=,line width=1.0pt,anchor=south west] (text6) at (-3.71, 
  69.152){$*$};

\end{tikzpicture}
\caption{Basis of $P_{k,+}$}\label{fig:basis-p-k-plus}
\end{subfigure}
\begin{subfigure}{.24\linewidth}
\centering
\definecolor{cc8c8c8}{RGB}{200,200,200}

\def \globalscale {1.000000}
\begin{tikzpicture}[y=1pt, x=1pt, yscale=\globalscale,xscale=\globalscale, every node/.append style={scale=\globalscale}, inner sep=0pt, outer sep=0pt]
  \path[draw=black,fill=cc8c8c8,line cap=,line width=1.0pt] (246.859, 47.0) 
  ellipse (22.059pt and 22.059pt);

  \path[draw=black,fill=white,line cap=,line width=1.0pt] (237.828, 53.672) 
  rectangle (257.172, 40.328);

  \node[line cap=,line width=1.0pt,anchor=south west] (text7) at (243.79, 
  44.152){$x$};

\end{tikzpicture}
\caption{Basis of $P_{0,-}$}
\end{subfigure}
\begin{subfigure}{.24\linewidth}
\centering
\def \globalscale {1.000000}
\begin{tikzpicture}[y=1pt, x=1pt, yscale=\globalscale,xscale=\globalscale, every node/.append style={scale=\globalscale}, inner sep=0pt, outer sep=0pt]
  \path[draw=black,line cap=,line width=1.0pt] (306.859, 47.0) ellipse (22.059pt
   and 22.059pt);

\end{tikzpicture}
\caption{Generator of $P_{0,+}$}
\end{subfigure}
\caption{Basis of the spin planar algebra}\label{fig:basis-spin-pl-alg}
\end{figure}

These generators satisfy the relation as in Figure~\ref{fig:spin-pl-alg-rels}, so that the loop parameter for this shaded planar algebra takes value $\sqrt{N}$.
\begin{figure}[h]
\begin{subfigure}{.3\linewidth}
\centering

\definecolor{cc8c8c8}{RGB}{200,200,200}

\def \globalscale {1.000000}
\begin{tikzpicture}[y=1pt, x=1pt, yscale=\globalscale,xscale=\globalscale, every node/.append style={scale=\globalscale}, inner sep=0pt, outer sep=0pt]
  \path[draw=black,fill=cc8c8c8,line cap=,line width=1.0pt] (247.0, 107.0) 
  circle (17.647pt);

  \path[draw=black,fill=white,line cap=,line width=1.0pt] (237.828, 113.672) 
  rectangle (257.172, 100.328);

  \node[line cap=,line width=1.0pt,anchor=south west] (text8) at (243.79, 
  104.152){$x$};

  \path[draw=black,dashed,line cap=,line width=0.978pt] (225.489, 129.511) rectangle 
  (269.511, 85.489);

  \path[draw=black,dashed,line cap=,line width=1.0pt] (304.717, 117.283) rectangle 
  (324.283, 97.717);

  \node[line cap=,line width=1.0pt,anchor=south west] (text10) at (276.79, 
  99.652){$=\frac{1}{\sqrt{N}}$};

\end{tikzpicture}

\end{subfigure}
\begin{subfigure}{.3\linewidth}
\centering

\definecolor{cc8c8c8}{RGB}{200,200,200}

\def \globalscale {1.000000}
\begin{tikzpicture}[y=1pt, x=1pt, yscale=\globalscale,xscale=\globalscale, every node/.append style={scale=\globalscale}, inner sep=0pt, outer sep=0pt]
  \path[draw=black,dashed,fill=cc8c8c8,line cap=,line width=0.978pt] (225.489, 189.511)
   rectangle (269.511, 145.489);

  \path[draw=black,fill=white,line cap=,line width=1.0pt] (247.0, 167.0) circle 
  (17.647pt);

  \path[draw=black,dashed,fill=cc8c8c8,line cap=,line width=1.0pt] (304.717, 177.283) 
  rectangle (324.283, 157.717);

  \node[line cap=,line width=1.0pt,anchor=south west] (text13) at (275.29, 
  162.652){$=\sqrt{N}$};

\end{tikzpicture}

\end{subfigure}
\begin{subfigure}{.3\linewidth}
\centering
\definecolor{cc8c8c8}{RGB}{200,200,200}

\def \globalscale {1.000000}
\begin{tikzpicture}[y=1pt, x=1pt, yscale=\globalscale,xscale=\globalscale, every node/.append style={scale=\globalscale}, inner sep=0pt, outer sep=0pt]
  \path[draw=black,fill=cc8c8c8,line cap=,line width=1.0pt] (246.859, 237.559) 
  ellipse (22.059pt and 22.059pt);

  \path[draw=black,fill=white,line cap=,line width=1.0pt] (237.828, 244.231) 
  rectangle (257.172, 230.887);

  \node[line cap=,line width=1.0pt,anchor=south west] (text14) at (243.79, 
  234.711){$x$};

  \path[draw=black,fill=cc8c8c8,line cap=,line width=1.0pt] (306.859, 237.559) 
  ellipse (22.059pt and 22.059pt);

  \node[line cap=,line width=1.0pt,anchor=south west] (text15) at (272.29, 
  234.711){$=$};

  \node[line cap=,line width=1.0pt,anchor=south west] (text16) at (207.79, 
  224.386){$\displaystyle\sum_x$};

\end{tikzpicture}
\end{subfigure}

\vspace{20pt}
\begin{subfigure}{.45\linewidth}
\centering

\definecolor{cc8c8c8}{RGB}{200,200,200}

\def \globalscale {1.000000}
\begin{tikzpicture}[y=1pt, x=1pt, yscale=\globalscale,xscale=\globalscale, every node/.append style={scale=\globalscale}, inner sep=0pt, outer sep=0pt]
  \path[draw=black,line cap=,line width=0.996pt,shift={(19.5, -0.0)}] (34.171, 
  159.502) -- (94.829, 159.502){[rotate=-45.0] arc(225.0:315.0:14.173 and 
  -14.173)} -- (109.002, 104.171){[rotate=-135.0] arc(225.0:315.0:14.173 and 
  -14.173)} -- (34.171, 89.998){[rotate=-45.0] arc(45.0:135.0:14.173 and 
  -14.173)} -- (19.998, 145.329){[rotate=-135.0] arc(45.0:135.0:14.173 and 
  -14.173)} -- cycle;

  \path[draw=black,fill=cc8c8c8,line cap=,line width=1.0pt] (113.512, 159.512)..
   controls (113.512, 143.483) and (100.299, 130.488) .. (84.0, 130.488).. 
  controls (67.701, 130.488) and (54.488, 143.483) .. (54.488, 159.512).. 
  controls (74.163, 159.512) and (93.837, 159.512) .. (113.512, 159.512) -- 
  cycle;

  \path[draw=black,fill=cc8c8c8,line cap=,line width=1.0pt] (113.512, 90.0).. 
  controls (113.512, 106.03) and (100.299, 119.024) .. (84.0, 119.024).. 
  controls (67.701, 119.024) and (54.488, 106.03) .. (54.488, 90.0).. controls 
  (74.163, 90.0) and (93.837, 90.0) .. (113.512, 90.0) -- cycle;

  \path[draw=black,fill=white,line cap=,line width=1.0pt] (74.328, 107.344) 
  rectangle (93.672, 94.0);

  \node[line cap=,line width=1.0pt,anchor=south west] (text21) at (80.29, 
  97.825){$x$};

  \path[draw=black,fill=white,line cap=,line width=1.0pt] (74.328, 155.844) 
  rectangle (93.672, 142.5);

  \node[line cap=,line width=1.0pt,anchor=south west] (text22) at (80.29, 
  146.325){$x$};

  \node[line cap=,line width=1.0pt,anchor=south west] (text23) at (2.29, 
  110.386){$\displaystyle\sum_x\sqrt{N}$};

  \node[line cap=,line width=1.0pt,anchor=south west] (text24) at (135.79, 
  120.711){$=$};

  \node[line cap=,line width=1.0pt,anchor=south west] (text9) at (35.29, 
  159.152){$*$};

  \path[draw=black,line cap=,line width=1pt,cm={ 
  0.663,-0.0,-0.0,0.656,(65.742, 42.341)}] (148.101, 159.572) -- (185.399, 
  159.572){[rotate=-45.0] arc(225.0:315.0:14.173 and -14.173)} -- (199.572, 
  104.102){[rotate=-135.0] arc(225.0:315.0:14.173 and -14.173)} -- (148.101, 
  89.928){[rotate=-45.0] arc(45.0:135.0:14.173 and -14.173)} -- (133.928, 
  145.399){[rotate=-135.0] arc(45.0:135.0:14.173 and -14.173)} -- cycle;

  \path[draw=black,fill=cc8c8c8,line cap=,line width=1.0pt] (167.47, 147.171) 
  rectangle (185.032, 101.5);

  \node[line cap=,line width=1.0pt,anchor=south west] (text34) at (147.79, 
  145.653){$*$};

\end{tikzpicture}
\end{subfigure}
\begin{subfigure}{.45\linewidth}
\centering

\definecolor{cc8c8c8}{RGB}{200,200,200}

\def \globalscale {1.000000}
\begin{tikzpicture}[y=1pt, x=1pt, yscale=\globalscale,xscale=\globalscale, every node/.append style={scale=\globalscale}, inner sep=0pt, outer sep=0pt]
  \path[draw=black,dashed,fill=cc8c8c8,line cap=,line width=1.0pt] (36.858, 259.617).. 
  controls (24.676, 259.617) and (14.801, 249.741) .. (14.801, 237.558).. 
  controls (14.801, 225.375) and (24.676, 215.501) .. (36.858, 215.501) -- 
  (57.858, 215.501).. controls (70.041, 215.501) and (79.917, 225.375) .. 
  (79.917, 237.558).. controls (79.917, 249.741) and (70.041, 259.617) .. 
  (57.858, 259.617) -- cycle;

  \path[draw=black,fill=white,line cap=,line width=1.0pt] (24.828, 244.231) 
  rectangle (44.172, 230.887);

  \node[line cap=,line width=1.0pt,anchor=south west] (text26) at (30.79, 
  234.711){$x$};

  \path[draw=black,dashed,fill=cc8c8c8,line cap=,line width=1.0pt] (139.559, 237.559) 
  ellipse (22.059pt and 22.059pt);

  \path[draw=black,fill=white,line cap=,line width=1.0pt] (130.0, 244.231) 
  rectangle (149.344, 230.887);

  \node[line cap=,line width=1.0pt,anchor=south west] (text28) at (135.963, 
  234.711){$x$};

  \path[draw=black,fill=white,line cap=,line width=1.0pt] (51.828, 244.231) 
  rectangle (71.172, 230.887);

  \node[line cap=,line width=1.0pt,anchor=south west] (text29) at (57.79, 
  233.211){$y$};

  \node[line cap=,line width=1.0pt,anchor=south west] (text27) at (86.29, 
  234.711){$=\delta_{x,y}$};

\end{tikzpicture}

\end{subfigure}

\caption{Relations in the spin planar algebra}\label{fig:spin-pl-alg-rels}
\end{figure}
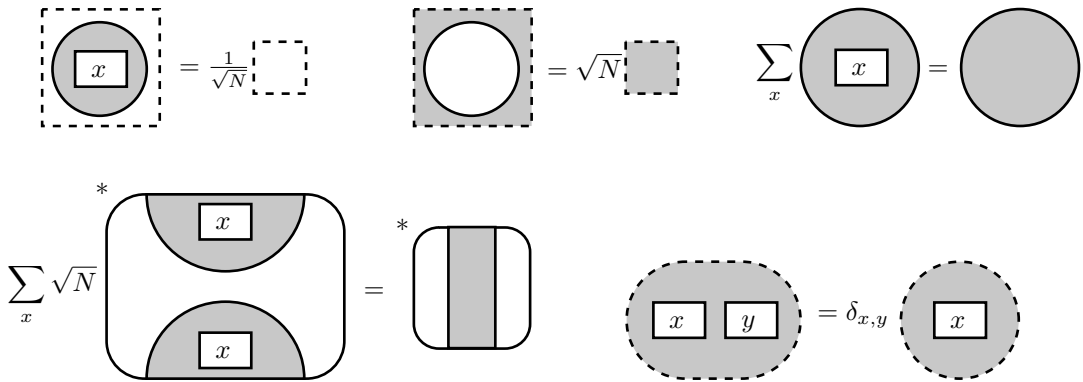

We can represent matrices $S \in M_N(\bC)$ by an element of $P_{2,+}$ as in Figure~\ref{fig:mat-diag-pres}, so that the matrix product agrees with the composition product in the planar algebra.
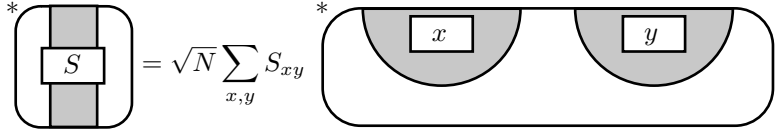
\begin{figure}[h]

\definecolor{cc8c8c8}{RGB}{200,200,200}

\def \globalscale {1.000000}
\begin{tikzpicture}[y=1pt, x=1pt, yscale=\globalscale,xscale=\globalscale, every node/.append style={scale=\globalscale}, inner sep=0pt, outer sep=0pt]
  \path[draw=black,line cap=,line width=1pt,cm={ 
  0.663,-0.0,-0.0,0.656,(-73.758, 300.341)}] (148.101, 159.572) -- (185.399, 
  159.572){[rotate=-45.0] arc(225.0:315.0:14.173 and -14.173)} -- (199.572, 
  104.102){[rotate=-135.0] arc(225.0:315.0:14.173 and -14.173)} -- (148.101, 
  89.928){[rotate=-45.0] arc(45.0:135.0:14.173 and -14.173)} -- (133.928, 
  145.399){[rotate=-135.0] arc(45.0:135.0:14.173 and -14.173)} -- cycle;

  \path[draw=black,fill=cc8c8c8,line cap=,line width=1.0pt] (27.97, 405.171) 
  rectangle (45.532, 359.5);

  \node[line cap=,line width=1.0pt,anchor=south west] (text11) at (11.29, 
  402.153){$*$};

  \path[draw=black,line cap=,line width=1.0pt,shift={(99.0, 79.5)}] (45.134, 
  324.736) -- (185.366, 324.736){[rotate=-45.0] arc(225.0:315.0:14.173 and 
  -14.173)} -- (199.539, 294.634){[rotate=-135.0] arc(225.0:315.0:14.173 and 
  -14.173)} -- (45.134, 280.461){[rotate=-45.0] arc(45.0:135.0:14.173 and 
  -14.173)} -- (30.961, 310.563){[rotate=-135.0] arc(45.0:135.0:14.173 and 
  -14.173)} -- cycle;

  \path[draw=black,fill=cc8c8c8,line cap=,line width=1.0pt] (204.012, 404.111)..
   controls (204.012, 388.081) and (190.799, 375.086) .. (174.5, 375.086).. 
  controls (158.201, 375.086) and (144.988, 388.081) .. (144.988, 404.111).. 
  controls (164.663, 404.111) and (184.337, 404.111) .. (204.012, 404.111) -- 
  cycle;

  \path[draw=black,fill=cc8c8c8,line cap=,line width=1.0pt] (283.512, 404.111)..
   controls (283.512, 388.081) and (270.299, 375.086) .. (254.0, 375.086).. 
  controls (237.701, 375.086) and (224.488, 388.081) .. (224.488, 404.111).. 
  controls (244.163, 404.111) and (263.837, 404.111) .. (283.512, 404.111) -- 
  cycle;

  \path[draw=black,fill=white,line cap=,line width=1.0pt] (162.859, 401.239) 
  rectangle (186.141, 387.957);

  \path[draw=black,fill=white,line cap=,line width=1.0pt] (242.359, 401.239) 
  rectangle (265.641, 387.957);

  \node[line cap=,line width=1.0pt,anchor=south west] (text17) at (170.79, 
  391.751){$x$};

  \node[line cap=,line width=1.0pt,anchor=south west] (text18) at (250.29, 
  390.251){$y$};

  \node[line cap=,line width=1.0pt,anchor=south west] (text19) at (127.29, 
  402.153){$*$};

  \node[line cap=,line width=1.0pt,anchor=south west] (text20) at (62.29, 
  366.886){$=\sqrt{N}\displaystyle\sum_{x, y}S_{x y}$};

  \path[draw=black,fill=white,line cap=,line width=1.0pt] (24.859, 389.239) 
  rectangle (48.141, 375.957);

  \node[line cap=,line width=1.0pt,anchor=south west] (text25) at (32.79, 
  379.751){$S$};

\end{tikzpicture}
\caption{Diagrammatic presentation of matrices}\label{fig:mat-diag-pres}
\end{figure}

Next, given a graph $X$ with $N$ vertices, and out of the Q-system $(Q, \phi)$ in $\Rep \Aut^+(X)$ as before, we get a \emph{general planar algebra} $(P^{X,g}_k)_{k=0}^\infty$ in the sense of~\cite{arXiv:math/9909027} by setting
\begin{align*}P^{X,g}_{2k+1} &= 0,&
P^{X,g}_{2 k} &= \Hom_{\Aut^+(X)}(\bC, Q^{\otimes k}).
\end{align*}
Following~\cite{MR2146039} (with an appropriate fix of its normalization errors), we can realize $P^{X, g}$ as the planar subalgebra of the spin planar algebra (note that the degree convention gives $P^{X, g}_{2 k} \subset P_{k,+}$) by the element of Figure~\ref{fig:adj-mat-elem}, which is the element corresponding to the adjacency matrix of $X$ up to the scaling by $\sqrt{N}$.
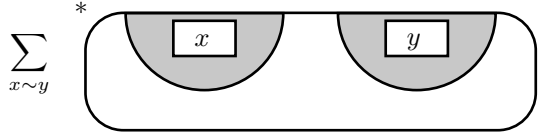
\begin{figure}[h]

\definecolor{cc8c8c8}{RGB}{200,200,200}

\def \globalscale {1.000000}
\begin{tikzpicture}[y=1pt, x=1pt, yscale=\globalscale,xscale=\globalscale, every node/.append style={scale=\globalscale}, inner sep=0pt, outer sep=0pt]
  \path[draw=black,line cap=,line width=1.0pt] (45.134, 324.736) -- (185.366, 
  324.736){[rotate=-45.0] arc(225.0:315.0:14.173 and -14.173)} -- (199.539, 
  294.634){[rotate=-135.0] arc(225.0:315.0:14.173 and -14.173)} -- (45.134, 
  280.461){[rotate=-45.0] arc(45.0:135.0:14.173 and -14.173)} -- (30.961, 
  310.563){[rotate=-135.0] arc(45.0:135.0:14.173 and -14.173)} -- cycle;

  \path[draw=black,fill=cc8c8c8,line cap=,line width=1.0pt] (105.012, 324.611)..
   controls (105.012, 308.581) and (91.799, 295.586) .. (75.5, 295.586).. 
  controls (59.201, 295.586) and (45.988, 308.581) .. (45.988, 324.611).. 
  controls (65.663, 324.611) and (85.337, 324.611) .. (105.012, 324.611) -- 
  cycle;

  \path[draw=black,fill=cc8c8c8,line cap=,line width=1.0pt] (184.512, 324.611)..
   controls (184.512, 308.581) and (171.299, 295.586) .. (155.0, 295.586).. 
  controls (138.701, 295.586) and (125.488, 308.581) .. (125.488, 324.611).. 
  controls (145.163, 324.611) and (164.837, 324.611) .. (184.512, 324.611) -- 
  cycle;

  \path[draw=black,fill=white,line cap=,line width=1.0pt] (63.859, 321.739) 
  rectangle (87.141, 308.457);

  \path[draw=black,fill=white,line cap=,line width=1.0pt] (143.359, 321.739) 
  rectangle (166.641, 308.457);

  \node[line cap=,line width=1.0pt,anchor=south west] (text32) at (71.79, 
  312.251){$x$};

  \node[line cap=,line width=1.0pt,anchor=south west] (text33) at (151.29, 
  310.751){$y$};

  \node[line cap=,line width=1.0pt,anchor=south west] (text35) at (26.79, 
  324.251){$*$};

  \node[line cap=,line width=1.0pt,anchor=south west] (text36) at (2.29, 
  293.386){$\displaystyle\sum_{x \sim y}$};

\end{tikzpicture}
\caption{Generator of $P^X$}\label{fig:adj-mat-elem}
\end{figure}
We write the associated shaded planar algebra as $P^X_{k,\pm}$.

\begin{proposition}\label{prop:P-Q-u-gen-t-d}
The C$^*$-tensor category $\Rep \Aut^+(X)$ is generated by the unit morphism $v\colon \bC \to Q$, the product morphism $m\colon Q \otimes Q \to Q$, and the morphism $d \colon Q \to Q$.
\end{proposition}

\begin{proof}
This is a consequence of the above characterization of $P^{X,g}$.
We can also directly obtain this starting from these three morphisms, by noting that $R = m^* v$ becomes a morphism of duality for $Q$ with $\bar R = R$.
\end{proof}

The C$^*$-Frobenius algebra object $Q$ in the rigid C$^*$-tensor category $\cC = \Rep \Aut^+(X)$ defines a C$^*$-$2$-category
\begin{align}\label{eq:2-cat-from-Q-sys}
\cC_{0, 0} &= \cC, &
\cC_{0, 1} &= \Modcat_\cC\mhyph Q, &
\cC_{1, 0} &= Q\mhyph\Modcat_\cC, &
\cC_{1,1} &= Q\mhyph\Modcat_\cC\mhyph Q.
\end{align}
Let $M \in \cC_{0, 1}$ denote $Q$ itself as a right $Q$-module, and $\bar M \in \cC_{1, 0}$ denote $Q$ itself as a left $Q$-module, which are dual to each other.
We thus have
\begin{align*}
P^X_{k, +} &= \End(M \otimes \bar M \otimes \dots),&
P^X_{k, -} = \End(\bar M \otimes M \otimes \dots),
\end{align*}
both with $k$ factors.

\subsection{3-transitive graphs}

We look at the \emph{3-transitive} (or \emph{3-homogeneous}) graphs, that guarantee $\dim P^X_{2,+} = 3$ and $\dim P^X_{3,+} \le 15$ (see Proposition~\ref{prop:d2d3}).
These graphs are classified by~\cite{MR805453}*{Corollary 1.2}, as follows:
\begin{enumerate}
\item the complete graphs $\rK_N$;
\item\label{it:disj-uni-Kn} disjoint union of copies of $\rK_n$ for a fixed $n$;
\item the pentagon;
\item the Higman--Sims graph;
\item the Hamming graphs $\rH(2, m)$ (AKA lattice graphs $\rL_2(m)$, rook graphs);
\item\label{it:class-orth-pol} the orthogonal polar graphs $\rO^-(6,q)$~\cite{brouwer-webpage-srg}, $q$ power of a prime; 
\item the affine polar graphs $\VO^+(2 k, 2)$~\cite{brouwer-webpage-srg} (`pseudo Latin square type');
\item the affine polar graphs $\VO^-(2 k, 2)$~\cite{brouwer-webpage-srg} (`negative Latin square type'); and
\item the McLaughlin graph.
\end{enumerate}

\begin{remark}
In~\cite{MR805453}, the series~\eqref{it:class-orth-pol} is described as the graphs `whose vertices are the maximal totally singular subspaces of the unitary space on $\PG(3,q)$'.
In the notation of~\cite{MR4350112} these are $\Delta(\mathsf{U}_4(q))$ while $\rO^-(6, q)$ are $\Gamma(\mathsf{O}^-_6(q))$, hence these are isomorphic by~\cite{MR4350112}*{Section 2.7.5}.
\end{remark}

\section{Main theorem}

\begin{theorem}\label{thm:main-thm}
Let $X$ be a $3$-transitive graph, not of the form $\rO^-(6, q)$ for $q > 3$.
Then we have one of the following.
\begin{enumerate}
\item\label{it:main-thm-no-symm-case}
There is no quantum symmetry of $X$, or equivalently, $\cO(\Aut^+(X))$ is commutative.
This happens for the pentagon, the graphs $\rO^-(6, 2)$ and $\rO^-(6, 3)$ and the McLaughlin graph.
\item\label{it:main-thm-Lie-case}
The representation category $\Rep \Aut^+(X)$ is unitarily monoidally equivalent to $\Rep G_q$ for $G = \PO(n)$ or $G = \PSp(n)$, and for some $n$ and $q > 0$.
In particular $\cO(\Aut^+(X))$ is noncommutative and infinite dimensional.
This happens for the Higman--Sims graph, the graphs $\rK_N$, $\rH(2, m)$, and $\VO^\epsilon(2 k, 2)$.
\item The graph $X$ is the disjoint union of $m$ copies of $\rK_n$ and we have $\Aut^+(X) = \rS^+_n \wr_* \rS^+_m$.
\end{enumerate}
\end{theorem}

We have the following conjecture for the 3-transitive graphs excluded in the theorem above. 

\begin{conjecture}
The graphs $\rO^-(6, q)$ for $q > 3$ do not have quantum symmetry.
\end{conjecture}

Our conjecture is based on the fact that several steps in the proof of no quantum symmetry for $q=2$ (Theorem \ref{thm:orth-pol-graph-no-q-sym}) work for general $q$, yet we could not make the proof work for the general case. 

We will summarize now what is already known about the quantum automorphism groups of $3-$transitive graphs. The case (3) in Theorem \ref{thm:main-thm} is the only non-connected case for which the quantum automorphism group was computed by Bichon~\cite{MR2096666}.

\begin{example}
Let $X$ be the Clebsch graph, which is isomorphic to $\VO^-(4,2)$.
As it has a square number of vertices, is triangle-free, and $3$-point regular, by the argument of~\cite{MR1469634} (and~\cite{MR1403861}) 
gives that $P^X$ is isomorphic (forgetting the unitary structure) to the shaded planar algebra of the $2$-category
\begin{align*}
\cC_{0, 0} &= \Rep \PSp_q(4) = \cC_{1,1}, &
\cC_{0, 1} &= (\Rep \Sp_q(4))_1 = \cC_{1, 0},
\end{align*}
for some $q$ and the `vector representation' object $V \in \cC_{0, 1}$, where $(\Rep \Sp_q(4))_1$ is the `odd' degree component of the $\Z_2$-extension structure over $\Rep \PSp_q(4)$, i.e., we have
\begin{equation}\label{eq:rep-spin-5-decomp}
\Rep \Sp_q(4) \simeq \Rep \PSp_q(4) \oplus (\Rep \Sp_q(4))_1
\end{equation}
corresponding to the fact that $\Sp(4)$ has $\Z_2$ as its center, with $\PSp(4)$ as the quotient by that.
Under this correspondence we should have $Q \simeq V \otimes \bar V \simeq V \otimes V$, while $d(V) = [4]_q$.
(We can also say $\Sp_q(4) = \Spin_q(5)$ covering $\PO(5) = \SO(5)$, and the above decomposition becomes
\[
\Rep \Spin_q(5) \simeq \Rep \SO_q(5) \oplus (\Rep \Spin_q(5))_1,
\]
where $V$ becomes the spin representation of $\Spin_q(4)$.)
As $X$ has $16$ vertices, we obtain $q = 1$, and we get the monoidal equivalence between $\Rep \Aut^+(X)$ and $\Rep \SO(5)$ which was obtained in~\cite{MR4117054} from an isomorphism of compact quantum groups between $\Aut^+(X)$ and $\SO_{-1}(5)$, a certain twist of $\SO(5) = \PSp(4)$.

Another approach is to use the skein relation~\cite{MR1469634}*{(3)} for $d$ to get $\dim P^X_{3,+} = 14$.
Then we obtain the unitary monoidal equivalence between $\Rep \Aut^+(X)$ and $\Rep \PSp(4)$ by~\cite{MR3592517}.
\end{example}

\begin{example}
Let $X$ be the Higman--Sims graph.
Again by the same reasoning as above, either via~\cite{MR1403861} and~\cite{MR1469634}, or via~\cite{MR1469634} and~\cite{MR3592517}, we obtain a unitary monoidal equivalence between $\Rep \Aut^+(X)$ and $\Rep \PSp_q(5)$, with $q$ satisfying $[4]_q = 10$.
This makes the relation alluded in~\cite{schmidt-thesis}*{Section 6.3} more precise.

This gives an example of Kac-type compact quantum group whose representation category has property (T) in the sense of~\citelist{\cite{MR3406647},~\cite{MR3509018},~\cite{MR3447719}}.
\end{example}

\begin{example}
For the Hamming graphs, we have $\Aut^+(\rH(2,n)) \simeq \rS_n^+ \wr \rS_2$~\cite{arXiv:2106.08787}.
They give categories $\Rep \PO_q(4)$.
These are related by $\Rep \rS_n^+ \simeq \Rep \SO_q(3)$ for $q + q^{-1} = n$, $\SO_q(3) \simeq \SU_q(2) / \Z_2$, and $\Pin_q(4) = (\SU_q(2) \times \SU_q(2)) \rtimes \Z_2$ covering $\PO_q(4)$ by a $4$-to-$1$ homomorphism.
\end{example}

\begin{example}
The graph $\VO^+(4, 2)$ is the complement of the Hamming graph $\rH(2, 2)$, hence we again see the connection to $\PO_q(4)$.
\end{example}

Thus, the remaining cases in Theorem~\ref{thm:main-thm} are: the case~\eqref{it:main-thm-no-symm-case} for $\rO^-(6, q)$ for $q=2,3$ (Theorem~\ref{thm:orth-pol-graph-no-q-sym} and Theorem~\ref{thm:McL-no-q-symm}) and the McLaughlin graph (Theorem~\ref{thm:McL-no-q-symm}), and the case~\eqref{it:main-thm-Lie-case} for the graphs $\VO^\epsilon(2 k, 2)$ (Corollaries~\ref{cor:VO-plus-PO-2-k} and~\ref{cor:VO-minus-PSp-2-k}).

\subsection{Planar algebras for the affine polar graphs}

The first few cases of $\VO^\epsilon(2 k, 2)$ can be studied through their planar algebras.
Specifically, we look at the \emph{Yang--Baxter type relations} in the sense of~\cite{arXiv:1507.06030}.
Recall that a shaded planar algebra $P$ satisfies the Yang--Baxter type relations when
\begin{itemize}
\item $\dim P_{i, \pm} = 1$ for $i = 0, 1$;
\item $\dim P_{2, \pm} = 3$, with an element $S \in P_{2, +}$ linearly indepdent from the Temperley--Lieb elements;
\item $\dim P_{3, \pm} = 15$; and
\item up to a linear span of elements involving planar diagrammatic calculus with at most $2$ copies of $S$, the elements in Figure~\ref{fig:YB-elements} agree.
\end{itemize}
Note that the last condition is independent on the choice of $S$.
A remarkable result of~\cite{arXiv:1507.06030} is that, if $S$ as above generates $P$ as a shaded planar algebra, then $P$ agrees with the one from Kauffmann bracket presentations, whose paramaters can be concretely computed.

\begin{figure}[h]
\begin{subfigure}{.4\linewidth}
\centering
\definecolor{cc7c7c7}{RGB}{199,199,199}

\def \globalscale {1.000000}
\begin{tikzpicture}[y=1pt, x=1pt, yscale=\globalscale,xscale=\globalscale, every node/.append style={scale=\globalscale}, inner sep=0pt, outer sep=0pt]
\path[draw=black,fill=cc7c7c7,line width=1.0pt] (10.0, 60.0).. controls (10.0, 60.0) and (10.0, 40.0) .. (10.0, 30.0).. controls (10.0, 20.0) and (17.52, 15.0) .. (25.0, 15.0).. controls (32.48, 15.0) and (57.93, 15.11) .. (65.0, 15.0).. controls (72.07, 14.89) and (80.2, 19.54) .. (80.0, 30.0).. controls (79.8, 40.47) and (80.0, 60.0) .. (80.0, 60.0) -- (70.0, 60.0).. controls (70.0, 60.0) and (70.07, 39.57) .. (70.0, 35.0).. controls (69.93, 30.43) and (67.14, 25.0) .. (60.0, 25.0).. controls (52.86, 25.0) and (49.83, 30.19) .. (50.0, 35.0).. controls (50.17, 39.81) and (50.0, 60.0) .. (50.0, 60.0) -- (40.0, 60.0).. controls (40.0, 60.0) and (39.83, 39.81) .. (40.0, 35.0).. controls (40.17, 30.19) and (37.63, 25.0) .. (30.0, 25.0).. controls (22.37, 25.0) and (20.07, 29.95) .. (20.0, 35.0).. controls (19.93, 40.06) and (20.0, 60.0) .. (20.0, 60.0) -- cycle;

\path[draw=black,fill=white,line width=1.0pt] (5.0, 50.0) rectangle (25.0, 35.0);

\node[anchor=south west,line width=1.0pt] (text5) at (10.0, 39.0){$S$};

\path[draw=black,fill=white,line width=1.0pt] (35.0, 50.0) rectangle (55.0, 35.0);

\node[anchor=south west,line width=1.0pt] (text6) at (40.0, 39.0){$S$};

\path[draw=black,fill=white,line width=1.0pt] (65.0, 50.0) rectangle (85.0, 35.0);

\node[anchor=south west,line width=1.0pt] (text7) at (70.0, 39.0){$S$};

\node[anchor=south west,line width=1.0pt] (text8) at (-1.76, 59.16){$*$};

\path[draw=black,line width=1.0pt] (5.0, 60.0) -- (85.0, 60.0){[rotate=-45.0] arc(225.0:315.0:5.0 and -5.0)} -- (90.0, 10.0){[rotate=-135.0] arc(225.0:315.0:5.0 and -5.0)} -- (5.0, 5.0){[rotate=-45.0] arc(45.0:135.0:5.0 and -5.0)} -- (-0.0, 55.0){[rotate=-135.0] arc(45.0:135.0:5.0 and -5.0)} -- cycle;

\end{tikzpicture}

\end{subfigure}
\begin{subfigure}{.45\linewidth}
\centering

\definecolor{cc7c7c7}{RGB}{199,199,199}

\def \globalscale {1.000000}
\begin{tikzpicture}[y=1pt, x=1pt, yscale=\globalscale,xscale=\globalscale, every node/.append style={scale=\globalscale}, inner sep=0pt, outer sep=0pt]
\path[draw=black,line width=1.0pt,cm={ 1.0,-0.0,-0.0,0.94,(-0.0, 5.07)}] (5.0, 85.0) -- (135.0, 85.0){[rotate=-45.0] arc(225.0:315.0:5.0 and -5.0)} -- (140.0, 5.0){[rotate=-135.0] arc(225.0:315.0:5.0 and -5.0)} -- (5.0, -0.0){[rotate=-45.0] arc(45.0:135.0:5.0 and -5.0)} -- (-0.0, 80.0){[rotate=-135.0] arc(45.0:135.0:5.0 and -5.0)} -- cycle;

\path[draw=black,fill=cc7c7c7,line width=1.0pt] (20.0, 85.0) -- (20.0, 35.0).. controls (20.0, 30.0) and (25.29, 25.0) .. (30.0, 25.0).. controls (34.71, 25.0) and (40.0, 30.0) .. (40.0, 35.0) -- (40.0, 65.0).. controls (40.0, 70.0) and (45.0, 75.0) .. (50.0, 75.0).. controls (55.0, 75.0) and (55.0, 75.0) .. (60.0, 75.0).. controls (65.0, 75.0) and (65.0, 75.0) .. (65.0, 80.0) -- (65.0, 85.0) -- (75.0, 85.0) -- (75.0, 80.0).. controls (75.0, 75.0) and (75.0, 75.0) .. (80.0, 75.0) -- (90.0, 75.0).. controls (95.0, 75.0) and (100.0, 70.0) .. (100.0, 65.5) -- (100.0, 35.0).. controls (100.0, 30.0) and (105.29, 25.0) .. (110.0, 25.0).. controls (114.71, 25.0) and (120.0, 30.0) .. (120.0, 35.0).. controls (120.0, 40.0) and (120.0, 85.0) .. (120.0, 85.0) -- (130.0, 85.0).. controls (130.0, 85.0) and (130.0, 37.0) .. (130.0, 30.0).. controls (130.0, 23.0) and (125.0, 15.0) .. (115.0, 15.0) -- (25.0, 15.0).. controls (15.0, 15.0) and (10.0, 23.0) .. (10.0, 30.0).. controls (10.0, 37.0) and (10.0, 85.0) .. (10.0, 85.0) -- cycle;

\path[draw=black,fill=white,line width=1.0pt] (50.0, 60.0).. controls (50.0, 65.0) and (51.82, 65.11) .. (55.0, 65.0).. controls (58.18, 64.89) and (82.4, 65.0) .. (85.0, 65.0).. controls (87.6, 65.0) and (90.0, 65.0) .. (90.0, 60.0).. controls (90.0, 55.0) and (90.0, 40.0) .. (90.0, 35.0).. controls (90.0, 30.0) and (84.5, 24.92) .. (80.0, 25.0).. controls (75.5, 25.08) and (64.5, 24.94) .. (60.0, 25.0).. controls (55.5, 25.06) and (50.0, 30.0) .. (50.0, 35.0).. controls (50.0, 40.0) and (50.0, 55.0) .. (50.0, 60.0) -- cycle;

\path[draw=black,fill=white,line width=1.0pt] (35.0, 54.5) rectangle (55.0, 39.5);

\node[anchor=south west,line width=1.0pt] (text9) at (40.0, 43.5){$S$};

\path[draw=black,fill=white,line width=1.0pt] (85.0, 54.5) rectangle (105.0, 39.5);

\node[anchor=south west,line width=1.0pt] (text10) at (90.0, 43.5){$S$};

\path[draw=black,fill=white,line width=1.0pt,rotate around={90.0:(0.0, 841.89)}] (-831.73, 779.05) rectangle (-811.73, 764.05);

\node[anchor=south west,line width=1.0pt,rotate=90.0] (text11) at (73.21, 15.16){$S$};

\node[anchor=south west,line width=1.0pt] (text12) at (-1.76, 84.66){$*$};

\end{tikzpicture}
\end{subfigure}
\caption{Elements of order $3$ in $P_{3,+}$}\label{fig:YB-elements}
\end{figure}
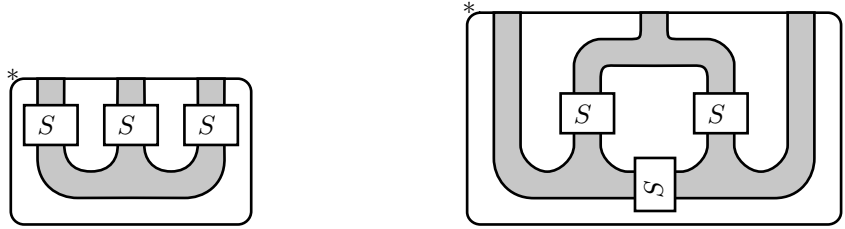

\begin{example}
Computing the parameters for the Yang--Baxter type relations for the planar algebra coming from the graph $\VO^+(6,2)$, we see that it corresponds to a category of the form $\Rep \PO_q(n)$.
Such a planar algebra can be presented as the Birman--Murakami--Wenzl planar algebra, with the value of a closed loop given by
\[
\delta = \frac{r - r^{-1}}{q - q^{-1}} + 1, \quad (r = q^{n-1}).
\]
The object $Q$ corresponds to $V^{\otimes 2}$ for the vector representation $V$ of $\rO_q(n)$ (and the product structure should be that of $\End(V)$ by the torsion-freeness of $\Rep \Spin_q(n)$; we can stick to the case $\absv{q} = 1$ by the discussion below).
We see that the braiding has the eigenvalues $1$ and $-1$, hence $q$ must be $1$ or $-1$.
This should correspond to a planar algebra corresponding to Brauer algebras, which appear as a limit of the BMW planar algebras.
We are at the limit case $\absv{q} = 1$, hence the above loop parameter becomes $n$.
We also know that this is $\sqrt{N}$, hence we get $n = \sqrt{N}$.
For $\VO^+(6,2)$, we have $N = 64$ hence $n = 8$.
Generally $\VO^+(2 k,2)$ has $N = 2^{2 k}$, hence one should expect $n = 2^k$.
\end{example}

\begin{example}
Turning to the graph $\VO^-(6,2)$, we see that the associated planar algebra again satisfies the Yang--Baxter type relations whose parameters correspond to the categories of the form \linebreak$\Rep \PSp_q(2 n)$.
Such a planar algebra can be presented as the BMW planar algebra with the value of a closed loop given by (up to sign)
\[
\delta = \frac{r - r^{-1}}{q - q^{-1}} + 1, \quad (r = q^{-2n-1}).
\]
The object $Q$ corresponds to $V^{\otimes 2}$ for the vector representation $V$ of $\Sp_q(2 n)$ (and the product structure is that of $\End(V)$ by the torsion-freeness of $\Rep \Sp_q(2 n)$ as before).
We see that the braiding has the eigenvalues $i$ and $-i$, hence $q$ must be $1$ or $-1$.
Again this should correspond to a planar algebra corresponding to Brauer algebras.
We are at the limit case $\absv{q} = 1$, hence the above loop parameter becomes $2 n$.
We also know that this is $\sqrt{N}$, hence we get $n = \frac{\sqrt{N}}2$.
For $\VO^-(6,2)$, we have $N = 64$ hence $n = 4$.
Generally $\VO^-(2 k,2)$ has $N = 2^{2 k}$, hence one should expect $n = 2^{k-1}$.
\end{example}

\section{The orthogonal polar graphs}

The orthogonal polar graph $\rO^{-}(6,q)$ is strongly regular with parameters $v=(q+1)(q^3+1)$, $k=q^3+q$, $\lambda =q-1$, and $\mu=q^2+1$.
(Following the standard convention we write $v$ instead of $N$ here.)
We are going to show that the graphs $\rO^{-}(6,q)$ do not have quantum symmetry for $q=2,3$. We handle the case $q=2$ in this section, the case $q=3$ will be obtained by a different method in the next section. Note that the graph $\rO^{-}(6,2)$ is the complement of the Schläfli graph. Therefore, we also obtain that the Schläfli graph has no quantum symmetry. 

Let us give a concrete presentation of these graphs.
Given a nonzero vector $v \in \rV(6, q) = \bF_q^6$, we denote the corresponding point in the projective space $\bP^5(\bF_q)$ by $\langle v \rangle$.
Let us take an anisotrophic quadratic form $Q_e$ on $\rV(2, q)$.
Concretely, it can be written as
\[
Q_e(v) = a_0 v_0^2 + a_1 v_0 v_1 + a_2 v_1^2 \quad (v = (v_0, v_1) \in \rV(2, q))
\]
when $a_0 + a_1 T + a_2 T^2$ is an irreducible quadratic polynomial over $\bF_q$.
Another interpretation is $Q_e(v) = N_{\bF_{q^2} / \bF_q}(v_0 + v_1 b)$ for the norm map $N_{\bF_{q^2} / \bF_q}(c) = c \sigma(c)$ of the extension $\bF_{q^2} / \bF_q$ ($\sigma$ being the nontrivial automorphism of $\bF_{q^2}$ over $\bF_q$), and a fixed $b \in \bF_{q^2} \setminus \bF_q$.

We then consider the nondegenerate quadratic form
\[
Q(v) = v_0 v_1 + v_2 v_3 + Q_e((v_4, v_5)) \quad (v = (v_i)_{0 \le i < 6} \in \rV(6, q))
\]
over $\rV(6, q)$.
By definition, the vertices of $\rO^-(6, q)$ are the points $\ang v$ of $\bP^5(\bF_q)$ such that $Q(v) = 0$, and two vertices $\ang v$ and $\ang w$ are adjacent if and only if $Q(v + w) = 0$.

Following~\cite{arXiv:1902.08984}, we consider the parameters $(q_i)_{i = 0}^3$ defined as follows.
Given $i = 0, 1, 2, 3$, take vertices $\ang u, \ang v, \ang w$ such that there are $i$ edges among them.
Then $q_i$ is the number of common neighbors of these vertices, see Figure~\ref{fig:q-i-conf}. The numbers $q_i$ are the same for any vertices $\ang u, \ang v, \ang w$ with $i$ edges among them since $\rO^{-}(6,q)$ is $3$-transitive and thus $3$-point regular. 

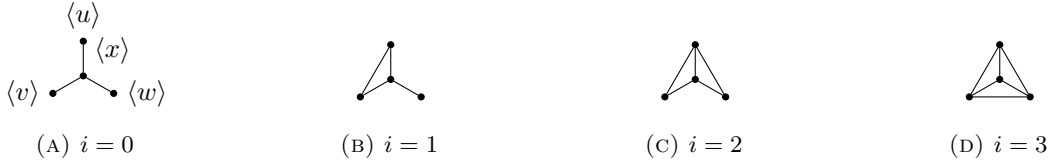
\begin{figure}[h]
\begin{subfigure}[b]{.245\linewidth}
\centering
\begin{tikzpicture}
\node[label={$\ang u$}] (corner0) at (0,1.73) [ext-vert] {};
\node[label=left:{$\ang v$}] (corner1) at (-1,0) [ext-vert] {};
\node[label=right:{$\ang w$}] (corner2) at (1,0) [ext-vert] {};
\node[label=60:$\ang x$] (inn_v0) at (0,0.58) [inn-vert] {};
\draw[] (corner0) -- (inn_v0);
\draw[] (corner1) -- (inn_v0);
\draw[] (corner2) -- (inn_v0);
\end{tikzpicture}
\caption{$i = 0$}
\end{subfigure}
\begin{subfigure}[b]{.245\linewidth}
\centering
\begin{tikzpicture}[baseline=-7pt]
\node (corner0) at (0,1.73) [ext-vert] {};
\node (corner1) at (-1,0) [ext-vert] {};
\node (corner2) at (1,0) [ext-vert] {};
\node (inn_v0) at (0,0.58) [inn-vert] {};
\draw[] (corner0) -- (corner1);
\draw[] (corner0) -- (inn_v0);
\draw[] (corner1) -- (inn_v0);
\draw[] (corner2) -- (inn_v0);
\end{tikzpicture}
\caption{$i = 1$}\label{q-i-subfig:2}
\end{subfigure}
\begin{subfigure}[b]{.245\linewidth}
\centering
\begin{tikzpicture}[baseline=-7pt]
\node (corner0) at (0,1.73) [ext-vert] {};
\node (corner1) at (-1,0) [ext-vert] {};
\node (corner2) at (1,0) [ext-vert] {};
\node (inn_v0) at (0,0.58) [inn-vert] {};
\draw[] (corner1) -- (corner0) -- (corner2);
\draw[] (corner0) -- (inn_v0);
\draw[] (corner1) -- (inn_v0);
\draw[] (corner2) -- (inn_v0);
\end{tikzpicture}
\caption{$i = 2$}\label{q-i-subfig:3}
\end{subfigure}
\begin{subfigure}[b]{.245\linewidth}
\centering
\begin{tikzpicture}[baseline=-7pt]
\node (corner0) at (0,1.73) [ext-vert] {};
\node (corner1) at (-1,0) [ext-vert] {};
\node (corner2) at (1,0) [ext-vert] {};
\node (inn_v0) at (0,0.58) [inn-vert] {};
\draw (corner0) -- (corner1) -- (corner2) -- (corner0);
\draw[] (corner0) -- (inn_v0);
\draw[] (corner1) -- (inn_v0);
\draw[] (corner2) -- (inn_v0);
\end{tikzpicture}
\caption{$i = 3$}
\end{subfigure}
\caption{$q_i$ counts the number of inner vertices $\ang x$}\label{fig:q-i-conf}
\end{figure}

\begin{lemma}\label{lem:orth-line-to-u}
Let $u = (1, 0, 0, 0, 0, 0) \in \rV(6, q)$, and suppose that $\ang u$ is connected to $\ang x$ in $\rO^-(6, q)$.
Then $x$ is of the form $(x_0, 0, x_2, x_3, x_4, x_5)$ with $x_2 x_3 + Q_e((x_4, x_5)) = 0$.
\end{lemma}

\begin{proof}
First, since $\ang x$ is an vertex of $\rO^-(6, q)$, we have
\begin{equation}\label{eq:x-def-vert-O-min-6-q}
x_0 x_1 + x_2 x_3 + Q((x_4, x_5)) = 0.
\end{equation}

Suppose $x_0 \neq 0$.
As $\ang{-x_0 u} = \ang u$ is adjacent to $\ang x$, we have $Q(x - x_0 u) = 0$, that is,
\[
x_2 x_3 + Q((x_4, x_5)) = 0.
\]
From this we get $x_0 x_1 = 0$, hence $x_1 = 0$.

Next suppose $x_0 = 0$.
Then using again that $\ang u$ is adjacent to $\ang x$, we have $Q(x + u) = 0$, that is,
\[
x_1 + x_2 x_3 + Q((x_4, x_5)) = 0.
\]
This forces $x_1 = 0$.
\end{proof}

\begin{proposition}
We have $q_0=q+1$, $q_1=1$, $q_2=0$ and $q_3=q-2$ for $\rO^-(6, q)$.
\end{proposition}

\begin{proof}
We will always use $u = (1, 0, 0, 0, 0, 0)$.
Then we consider $v$, $w$, and $x$ that will give above configurations.

\begin{step}
$q_3 = q-2$.
\end{step}

Here we take $v = (0, 0, 1, 0, 0, 0)$ and $w = (1, 0, 1, 0, 0, 0)$.
By Lemma~\ref{lem:orth-line-to-u} and its analogue for $v$, we see that $x$ is of the form $(x_0, 0, x_2, 0, x_4, x_5)$.
Then~\eqref{eq:x-def-vert-O-min-6-q} implies that $Q_e((x_4, x_5)) = 0$, hence $x_4 = 0 = x_5$.
It follows that $x$ is of the form $(x_0, 0, x_2, 0, 0, 0)$, and any such $x$ is allowed except for $x_0 = 0 = x_2$.
Since we are counting lines that are not $\ang u$, $\ang v$, nor $\ang w$, we have $(q^2 - 1) / (q - 1) - 3 = q - 2$ points $\ang x$ in this configuration.

\begin{step}
$q_2 = 0$.
\end{step}

Here we take $v = (0, 0, 1, 0, 0, 0)$ and $w = (0, 0, 0, 1, 0, 0)$.
Again $x$ is of the form $(x_0, 0, x_2, 0, 0, 0)$ from the relation between $u$ and $v$.
For $\ang x$ to be adjacent to $\ang w$, we must have $x_2 = 0$.
This forces $\ang x = \ang v$, which is forbidden.

\begin{step}
$q_1 = 1$.
\end{step}

Here we take $v = (0, 0, 1, 0, 0, 0)$ and $w = (0, 1, 0, 1, 0, 0)$, so that $x$ is of the form $(x_0, 0, x_2, 0, 0, 0)$.
We get $x_0 + x_2 = 0$ for $\ang x$ to be adjacent to $\ang w$, hence there is exactly one point $\ang x$ in this configuration.

\begin{step}
$q_0 = q+1$.
\end{step}

Here we take $v = (0, 1, 1, 0, 0, 0)$ and $w = (0, 1, 0, 1, 0, 0)$.
From $Q(x + v) = 0 = Q(x)$, we get $x_0 + x_2 = 0$.
Similarly we get $x_0 + x_3 = 0$ from $Q(x + w) = 0$, hence we have $x = (a, 0, -a, -a, x_4, x_5)$ with $a^2 + Q_e((x_4, x_5)) = 0$.
Now, the norm map is a surjective homomorphism $\bF_{q^2}^\times\to \bF_q^\times$, hence for each possible value of $-a^2$ there are $q+1$ solutions of $(x_4, x_5)$.

Suppose $q = 2^e$.
Then $\bF_q^\times$ is a cyclic group of odd order, hence $a^2 = -a^2$ can take all possible values.
Then there are $(q-1)(q+1)$ possible configurations of $x$, and since we are counting lines we have $q+1$ vertices $\ang x$.

Next suppose $q = p^e$ for an odd prime.
Then $\bF_q^\times$ is a cyclic group of even order, hence $-a^2$ can take $(q-1)/2$ different values.
This time, if $x = (a, 0, -a, -a, x_4, x_5)$ represents a valid solution, $(-a, 0, a, a, x_4, x_5)$ is another valid solution.
Thus we have $2 \times (q-1) / 2 \times (q+1)$ possible configurations of $x$, hence $q+1$ vertices $\ang x$ again.
\end{proof}



\begin{theorem}\label{thm:orth-pol-graph-no-q-sym}
The algebra $\cO(\rO^-(6, 2))$ is commutative. Therefore, the Schläfli graph has no quantum symmetry.
\end{theorem}

\begin{proof}
Since a major part of the proof is valid for $O^-(6, q)$ in general, let us start in that setting.
As before, let us denote by $u_{ij}$,  $0 \leq i,j < v$ the generators of $\cO(\Aut^+(\rO^{-}(6,q)))$.
We have orthogonality among some generators from~\eqref{rel3}, and also from $u_{i j} u_{k l} = 0$ if $\absv{\{i, j, k, l\}} = 3$, see~\eqref{rel1}.
Thus, it is enough to show $u_{ij}u_{kl}=u_{kl}u_{ij}$ for all $(i,k)\in E$, $(j,l)\in E$, and for all $(i,k)\notin E$, $(j,l)\notin E$.

\begin{step}\label{st:1}
We have $u_{ij}u_{kl}u_{ip}=0$ for $(i,k)\in E, (j,l)\in E$ and $p\neq j, (j,p)\notin E$.
\end{step}

Let $p\neq j, (j,p)\notin E$. Since $q_2=0$ there is no common neighbor of $j,l$ and $p$. Since $\lambda =q-1$, there is a common neighbor $b$ of $j$ and $l$, which is not adjacent to $p$ by the previous sentence. By~\eqref{rel2} and~\eqref{rel3}, we have
\begin{equation*}
u_{ij}u_{kl}u_{ip}=u_{ij}\biggl(\sum_{t;(i,t)\in E, (t,k)\in E}u_{tb}\biggr)u_{kl}u_{ip}. 
\end{equation*}
Since $q_1=1$, the vertex $l$ is the only common neighbor of $j,b$ and $p$. Let $a\neq l$. If $a$ is not adjacent to $b$ or $p$, we directly get 
\begin{align*}
u_{ij}\biggl(\sum_{t;(i,t)\in E, (t,k)\in E}u_{tb}\biggr)u_{ka}u_{ip}=0
\end{align*}
 by~\eqref{rel3}. If $a$ is adjacent to $p$ and $b$ but not adjacent to $j$, then 
 \begin{align*}
 u_{ij}\biggl(\sum_{t;(i,t)\in E, (t,k)\in E}u_{tb}\biggr)u_{ka}u_{ip}=u_{ij}u_{ka}u_{ip}=0
 \end{align*}
  by~\eqref{rel3}. Thus, we get 
  \begin{align*}
  u_{ij}u_{kl}u_{ip}&=u_{ij}\biggl(\sum_{t;(i,t)\in E, (t,k)\in E}u_{tb}\biggr)u_{kl}u_{ip}\\
  &=u_{ij}\biggl(\sum_{t;(i,t)\in E, (t,k)\in E}u_{tb}\biggr)\biggl(\sum_a u_{ka}\biggr)u_{ip}\\
  &=u_{ij}\biggl(\sum_{t;(i,t)\in E, (t,k)\in E}u_{tb}\biggr)u_{ip}\\
  &=0,
  \end{align*}
  since $(b,p)\notin E$. 

\begin{step}\label{st:3}
For $q=2$, it holds $u_{ij}u_{kl}=u_{kl}u_{ij}$ for $(i,k)\in E$ and $(j,l)\in E$.
\end{step}

Let $(i,k)\in E$ and $(j,l)\in E$. It holds
\begin{align*}
u_{ij}u_{kl}=u_{ij}u_{kl}\sum_p u_{ip}=u_{ij}u_{kl}\sum_{p;(l,p)\in E} u_{ip} 
\end{align*}
by first using~\eqref{rel2} and then~\eqref{rel3}. By Step~\ref{st:1}, we get 
\begin{align}
u_{ij}u_{kl}\sum_{p;(l,p)\in E} u_{ip} =u_{ij}u_{kl}\sum_{\substack{p;(l,p)\in E, (p,j)\in E,\\ \text{or } p=j}} u_{ip}.\label{eq3}
\end{align}
Consider $p\neq j$ such that $(l,p)\in E$ and $(p,j)\in E$. Since for $q=2$ it holds $\lambda=1$, we know that $l$ is the only common neighbor of $j$ and $p$. Therefore, we get
\begin{align*}
    u_{ij}u_{ka}u_{ip}=0
\end{align*}
 for $a\neq l$, since it holds $(j,a)\notin E$ or $(a,p)\notin E$ for all such $a$. We deduce
 \begin{align*}
     u_{ij}u_{kl}u_{ip}=u_{ij}\left(\sum_a u_{ka}\right)u_{ip}=u_{ij}u_{ip}=0,
 \end{align*}
 since $j\neq p$. Together with \eqref{eq3}, this yields $u_{ij}u_{kl}=u_{ij}u_{kl}u_{ij}$. Using the involution, we get the assertion of Step \eqref{st:3}.

\begin{step}\label{st:4}
It holds $u_{ij}u_{kl}=u_{kl}u_{ij}$ for $(i,k)\notin E$ and $(j,l)\notin E$.
\end{step}

Let $(i,k)\notin E$ and $(j,l)\notin E$. It holds
\begin{align*}
u_{ij}u_{kl}=u_{ij}u_{kl}\sum_p u_{ip}=u_{ij}u_{kl}\sum_{p;(l,p)\notin E} u_{ip}
\end{align*}
by first using~\eqref{rel2} and then~\eqref{rel3}. Let $p\neq j, (l,p)\notin E$. Let $s$ be a common neighbor of $j$ and $l$ that is not adjacent to $p$. Such $s$ exists because of the following. We know that $j,l$ have $\mu=q^2+1$ common neighbors. For $(p,j)\in E$ only one of those is adjacent to $p$ since $q_1=1$ and for $(p,j)\notin E$, there are $q+1$ common neighbors that are adjacent to $p$. This leaves us with $q^2$ and $q^2-q$ such $s$ for $(p,j)\in E$, $(p,j)\notin E$, respectively. We have 
\begin{align*}
u_{ij}u_{kl}u_{ip}= u_{ij}\biggl(\sum_{\substack{t;(t,i)\in E; (k,t)\in E}}u_{ts}\biggr)u_{kl}u_{ip}
\end{align*}
by~\eqref{rel2} and~\eqref{rel3}. By Step~\ref{st:3},
we get 
\begin{align*}
u_{ij}\biggl(\sum_{\substack{t;(t,i)\in E; (k,t)\in E}}u_{ts}\biggr)u_{kl}u_{ip}=u_{ij}u_{kl}\biggl(\sum_{\substack{t;(t,i)\in E; (k,t)\in E}}u_{ts}\biggr)u_{ip}=0
\end{align*}
because $u_{ts}u_{ip}=0$ for $(t,i)\in E, (s,p)\notin E$. We conclude $u_{ij}u_{kl}=u_{ij}u_{kl}u_{ij}$ for $(i,k)\notin E$ and $(j,l)\notin E$.

We conclude that the Schläfli graph has no quantum symmetry, since it is the complement of the graph $\rO^-(6, 2)$.

 \end{proof}

\section{Combinatorial presentation}\label{sec:comb-pres}

Our next goal is to show that the McLaughlin graph and the graph $\rO^-(6, 3)$ do not have quantum symmetry.
In terms of the planar algebra $P^X$, this is equivalent to showing that the `flip' element
\[
\sum_{0 \le x, y < N} \delta_x \otimes \delta_y \otimes \delta_x \otimes \delta_y \in P_{4,+}
\]
of the spin planar algebra belongs to $P^X_{4,+}$.

\subsection{Combinatorial intertwiners}

We use certain planar graphs to represent elements of $P^X_{k,+}$.
Such a graph should satisfy the following components:
\begin{itemize}
\item $k$ `external' vertices $v^e_i$ for $0 \le i < k$,
\item zero or more `internal' vertices,
\item double strike edges between external vertices,
\item single strike edges between vertices.
\end{itemize}
The planarity means that when we arrange the external vertices in clockwise order on a convex $k$-gon, the internal vertices can be arranged inside it such that the edges do not intersect.

Given such a graph $\Gamma$, consider the factors
\begin{itemize}
\item $c_\Gamma^{x_0,\dots,x_{k-1}}$: the number of embeddings of the subgraph of $\Gamma$ formed by the single strike edges (and all vertices) into $X$ that sends each $v^e_i$ to $x_i$,
\item $d_\Gamma^{x_0,\dots,x_{k-1}} \in \{0, 1\}$: $d_\Gamma^{x_0,\dots,x_{k-1}} = 1$ if, up to the above correspondence of vertices, the values of $x_i$ are constant on the connected components of the subgraph of $\Gamma$ formed by the single strike edges, and $d_\Gamma^{x_0,\dots,x_{k-1}} = 0$ otherwise.
\end{itemize}
The element $a_\Gamma$ of $P_{k,+}$ associated with $\Gamma$ is given by
\begin{equation}\label{eq:a-Gamma-element}
a_\Gamma = \sum_{\substack{0 \le x_i < N\\ 0 \le i < k}}  c_\Gamma^{x_0,\dots,x_{k-1}} d_\Gamma^{x_0,\dots,x_{k-1}} \delta_{x_0} \otimes \dots \otimes \delta_{x_{k-1}}.
\end{equation}

\begin{remark}
This is nothing but the element $T^{K \to X}$ of~\cite{MR4232075}, where $K$ is a certain (bi-)labeled graph whose underlying graph is obtained by identifying the external vertices of $\Gamma$ that are connected by double strike edges, and we take $\ell = 0$ in their notation.
In the presentation of planar algebra elements from Section~\ref{sec:cat-pres}, a black segment on the boundary corresponds to an external vertex.
A black connected component on the boundary correspond to a connected components by double strike edge.
A black region not touching the boundary correspond to an internal vertex.
A $2$-box (element of $P^X_{2,+}$) labeled by the adjacency matrix $A$ corresponds to a single strike edge.
\end{remark}

\subsection{Some reductions}

There are some simplifications we can make to understand the elements of $P^X_{k,+}$.
First is the following reduction to orbits.

\begin{proposition}
Let $a$ be an element of $P^X_{k,+}$, given by the linear combination of simple tensors
\begin{equation}\label{eq:elem-P-k-from-coeff}
a = \sum_{\substack{0 \le x_i < N\\ 0 \le i < k}} a^{x_0, \dots, x_{k-1}} \delta_{x_0} \otimes \cdots \otimes \delta_{x_{k-1}} \in P_{k,+}.
\end{equation}
For any $g \in \Aut(X)$, we have $a^{x_0, \dots, x_{k-1}}  = a^{g(x_0), \dots, g(x_{k-1})}$.
\end{proposition}

This is also obvious for the elements $a_\Gamma$ introduced above.

\begin{proof}
The inclusion of compact quantum groups $\Aut(X) < \Aut^+(X)$ corresponds to the inclusion of planar algebras $P^X < P^{\Aut(X)}$.
We have the claimed relation of coefficients for the elements in $P^{\Aut(X)}_{k,+} = (Q^{\otimes k})^{\Aut(X)}$.
\end{proof}

\begin{proposition}\label{prop:d2d3}
Let $X$ be a $3$-transitive graph.
Then we have
\begin{align*}
\dim P^X_{2,+} &\le 3,&
\dim P^X_{3,+} &\le 15.
\end{align*}
\end{proposition}

Note that we have $P^X_{2,+} = 2$ only in the case $X = K_N$.

\begin{proof}
Generally we have $\dim P^X_{k,+} \le \dim (Q^{\otimes k})^{\Aut(X)}$ from the inclusion of planar algebras.
Note that the right hand side is equal to the number $\omega_k$ of orbits of $\Aut(X)$ on $\{i \mid 0 \le i < N\}^k$.
For $k = 2$, we have $\omega_2 \le 3$ from the $2$-transitivity of $X$.
For $k = 3$, we have $15$ configurations of the triples of vertices, that bound $\omega_3$ by the $3$-transitivity of $X$, see Figure~\ref{fig:3-vert-conf}.

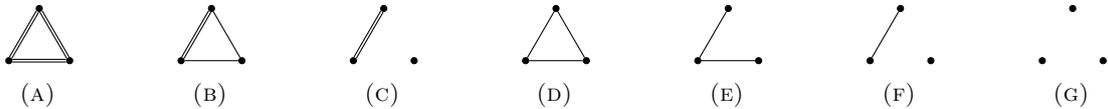
\begin{figure}[h]
\begin{subfigure}[b]{.135\linewidth}
\centering
\begin{tikzpicture}
\node (corner0) at (0,1.73) [ext-vert] {};
\node (corner1) at (-1,0) [ext-vert] {};
\node (corner2) at (1,0) [ext-vert] {};
\draw[double] (corner0) -- (corner1) -- (corner2) -- (corner0);
\end{tikzpicture}
\caption{}
\end{subfigure}
\begin{subfigure}[b]{.135\linewidth}
\centering
\begin{tikzpicture}
\node (corner0) at (0,1.73) [ext-vert] {};
\node (corner1) at (-1,0) [ext-vert] {};
\node (corner2) at (1,0) [ext-vert] {};
\draw[double] (corner0) -- (corner1);
\draw (corner1) -- (corner2) -- (corner0);
\end{tikzpicture}
\caption{}\label{subfig:2}
\end{subfigure}
\begin{subfigure}[b]{.135\linewidth}
\centering
\begin{tikzpicture}
\node (corner0) at (0,1.73) [ext-vert] {};
\node (corner1) at (-1,0) [ext-vert] {};
\node (corner2) at (1,0) [ext-vert] {};
\draw[double] (corner0) -- (corner1);
\end{tikzpicture}
\caption{}\label{subfig:3}
\end{subfigure}
\begin{subfigure}[b]{.135\linewidth}
\centering
\begin{tikzpicture}
\node (corner0) at (0,1.73) [ext-vert] {};
\node (corner1) at (-1,0) [ext-vert] {};
\node (corner2) at (1,0) [ext-vert] {};
\draw (corner0) -- (corner1) -- (corner2) -- (corner0);
\end{tikzpicture}
\caption{}
\end{subfigure}
\begin{subfigure}[b]{.135\linewidth}
\centering
\begin{tikzpicture}
\node (corner0) at (0,1.73) [ext-vert] {};
\node (corner1) at (-1,0) [ext-vert] {};
\node (corner2) at (1,0) [ext-vert] {};
\draw (corner0) -- (corner1) -- (corner2);
\end{tikzpicture}
\caption{}\label{subfig:5}
\end{subfigure}
\begin{subfigure}[b]{.135\linewidth}
\centering
\begin{tikzpicture}
\node (corner0) at (0,1.73) [ext-vert] {};
\node (corner1) at (-1,0) [ext-vert] {};
\node (corner2) at (1,0) [ext-vert] {};
\draw (corner0) -- (corner1);
\end{tikzpicture}
\caption{}\label{subfig:6}
\end{subfigure}
\begin{subfigure}[b]{.135\linewidth}
\centering
\begin{tikzpicture}
\node (corner0) at (0,1.73) [ext-vert] {};
\node (corner1) at (-1,0) [ext-vert] {};
\node (corner2) at (1,0) [ext-vert] {};
\end{tikzpicture}
\caption{}
\end{subfigure}
\caption{Configurations of triples of vertices; double strike edges represent equality}\label{fig:3-vert-conf}
\end{figure}

Note that the configurations~\ref{subfig:2},~\ref{subfig:3},~\ref{subfig:5}, and~\ref{subfig:6} occur with multiplicity $3$ because we are considering ordered triples of vertices.
\end{proof}

Let us next get other more combinatorial reductions under the assumption $\dim P^X_{2,+} = 3$.
We can exclude parallel edges with the same endpoints in $\Gamma$, which would correspond to the Schur product. (Having a simple graph $X$ also leads to this simplification.)

When $\dim P^X_{2,+} = 3$ happens, we can exclude the inner vertices of valency $2$, because that corresponds to $A^2$ which would be in the linear span of $A$ (that vertex removed and the edges joined), the matrix of all $1$'s (the vertex and the edges removed), and $I_N$ (the vertex removed, its neighbors identified).

\begin{proposition}
When the Yang--Baxter type relations fail in $P^X$, we can exclude the inner vertices of valency $3$ from $\Gamma$.
\end{proposition}

\begin{proof}
Suppose that $\Gamma$ has $k$ external vertices and an internal vertex $v_i$ of valency $3$.
We claim that $a_\Gamma$ is in the linear span of $a_{\Gamma'}$ with graphs $\Gamma'$ with $k$ external vertices that have fewer single strike edges.

Let $\Gamma_0$ be the graph with three outer vertices and one inner vertex of valency $3$, see Figure~\ref{fig:triv-inn-vert}.
\begin{figure}[h]
\begin{tikzpicture}
\node (corner0) at (0,1.73) [ext-vert] {};
\node (corner1) at (-1,0) [ext-vert] {};
\node (corner2) at (1,0) [ext-vert] {};
\node (inn_v0) at (0,0.58) [inn-vert] {};
\draw (corner0) -- (inn_v0);
\draw (corner1) -- (inn_v0);
\draw (corner2) -- (inn_v0);
\end{tikzpicture}
\caption{$\Gamma_0$}\label{fig:triv-inn-vert}
\end{figure}
By the failure of the Yang--Baxter type equation, the element $a_{\Gamma_0}$ is contained in the linear span of $a_{\Gamma''}$ with graphs $\Gamma''$ with three external vertices and no internal vertex, which are not the triangle on external vertices.

Suppose that $v'_j$ ($0 \le j < 2$) are connected to $v_i$ in $\Gamma$.
Up to the linear combination (and planar calculus), $a_\Gamma$ is in the linear span of $a_{\Gamma'}$ where $\Gamma'$ is obtained from $\Gamma$ by removing $v_i$, identifying some of the $v'_j$ (corresponding to double strike edges in $\Gamma''$), and adding at most two single strike edges between some of the $v'_j$ (corresponding to single strike edges in $\Gamma''$).
\end{proof}

\subsection{Detecting the lack of quantum symmetry}

By the previous section, to see that a $3$-transitive graph $X$ does not have quantum symmetry, we should use the graphs $\Gamma$ with inner vertices having valency $3$ or more to produce elements of $P^X_{4,+}$ to produce the flip element, or equivalently, establish $\dim P^X_{4,+} = \dim P^{\Aut(X)}_{4,+}$.
We achieve this by computer algebra, and explained in Appendix~\ref{sec:appendix}.
The linear equation for finding the flip is handled as linear algebra with rational coefficients, and the solutions are exact.

In the examples below, we write $a_j$ for the element $a_\Gamma \in P^X_{4,+}$ with $\Gamma = \Gamma_j$ in Table~\ref{tab:graphs}.

\begin{example}
In Theorem~\ref{thm:orth-pol-graph-no-q-sym}, we checked that the graph $X = \rO^-(6, 2)$ does not have quantum symmetry.
Concretely, the element
\begin{multline*}
2 a_{0} +  a_{1} +  a_{2} +  a_{3} +  a_{4} + 2 a_{5} + 2 a_{6} + \frac{3}{2} a_{7} + \frac{3}{2} a_{8} + \frac{3}{2} a_{9} + \frac{3}{2} a_{10} + 3 a_{11} - \frac{1}{4} a_{12} - \frac{1}{4} a_{13} - 2 a_{14} - 2 a_{15} - 2 a_{16} - 2 a_{17}\\
-  a_{18} -  a_{19} -  a_{20} -  a_{21} -  a_{22} -  a_{23} -  a_{24} -  a_{25} -  a_{26} -  a_{27} -  a_{28} -  a_{29} -  a_{30} -  a_{31} - \frac{3}{2} a_{32} - \frac{3}{2} a_{33} - \frac{3}{2} a_{34} - \frac{3}{2} a_{35} -  a_{36}\\
-  a_{37} -  a_{38} -  a_{39} + 2 a_{42} + 2 a_{43} + 2 a_{44} + 2 a_{45} +  a_{46} +  a_{47} +  a_{48} +  a_{49} +  a_{50} +  a_{51} +  a_{52} +  a_{53} + 2 a_{54} + 2 a_{55} +  a_{56} +  a_{57}\\
+  a_{58} +  a_{59} +  a_{60} +  a_{61} +  a_{62} +  a_{63} +  a_{64} +  a_{65} +  a_{66} +  a_{67} +  a_{68} +  a_{69} -  a_{70} -  a_{71} -  a_{72} -  a_{73} - 2 a_{74} - 2 a_{75} - 2 a_{76} - 2 a_{77}\\
-  a_{78} -  a_{79} -  a_{80} -  a_{81} -  a_{82} -  a_{83} -  a_{84} -  a_{85} -  a_{86} -  a_{87} -  a_{88} -  a_{89} +  a_{90} +  a_{91} +  a_{92} +  a_{93} +  a_{94} +  a_{95} +  a_{96} +  a_{97}\\
+ 2 a_{98} -  a_{99} - \frac{1}{2} a_{100} - \frac{1}{2} a_{101} +  a_{102} +  a_{103} +  a_{104} +  a_{105} - \frac{3}{2} a_{106} - \frac{3}{2} a_{107} - \frac{3}{2} a_{108} - \frac{3}{2} a_{109} -  a_{110} -  a_{111} + 6 a_{112}\\
- \frac{1}{2} a_{116} - \frac{1}{2} a_{117} + \frac{1}{4} a_{195}
\end{multline*}
represents the flip in $P^X_{4,+}$.
\end{example}

\begin{theorem}\label{thm:McL-no-q-symm}
Let $X$ be the McLaughlin graph or $\rO^-(6, 3)$.
Then $P^X_{4,+}$ contains the flip.
\end{theorem}

\begin{proof}
The dimension of $P^{\Aut(X)}_{4,+}$ is $128$ for the McLaughlin graph, and $134$ for the $\rO^-(6, 3)$.
We find enough diagrammatic elements to achieve this dimension, and also find an explicit formula for the flip element.

For the McLaughlin graph, the elements $a_i$ up to $i = 142$, together with the ones for $i = 195$, $202$, and $203$ span a $128$-dimensional space inside $P^X_{4,+}$, implying the claim.
Concretely, the element
\begin{multline*}
\frac{12865}{729} a_{0} - \frac{1753}{729} a_{1} - \frac{1753}{729} a_{2} - \frac{1753}{729} a_{3} - \frac{1753}{729} a_{4} + \frac{9301}{243} a_{5} + \frac{9301}{243} a_{6} + \frac{1760}{81} a_{7} + \frac{1760}{81} a_{8} + \frac{1760}{81} a_{9} + \frac{1760}{81} a_{10}\\
- \frac{37804}{243} a_{11} - \frac{13657}{729} a_{12} - \frac{13657}{729} a_{13} - \frac{12457}{1458} a_{14} - \frac{12457}{1458} a_{15} - \frac{12457}{1458} a_{16} - \frac{12457}{1458} a_{17} - \frac{4016}{243} a_{18} - \frac{4016}{243} a_{19}\\
+ \frac{79}{81} a_{20} + \frac{79}{81} a_{21} + \frac{79}{81} a_{22} + \frac{79}{81} a_{23} + \frac{271}{729} a_{24} + \frac{271}{729} a_{25} + \frac{271}{729} a_{26} + \frac{271}{729} a_{27} + \frac{1643}{1458} a_{28} + \frac{1643}{1458} a_{29} + \frac{1643}{1458} a_{30}\\
+ \frac{1643}{1458} a_{31} - \frac{17728}{729} a_{32} - \frac{17368}{729} a_{33} - \frac{17368}{729} a_{34} - \frac{17728}{729} a_{35} - \frac{2230}{729} a_{36} - \frac{50}{27} a_{37} - \frac{50}{27} a_{38} - \frac{2230}{729} a_{39} + \frac{3628}{243} a_{40}\\
+ \frac{3628}{243} a_{41} + \frac{1019}{243} a_{42} + \frac{1019}{243} a_{43} + \frac{1817}{729} a_{44} + \frac{1817}{729} a_{45} + \frac{11117}{1458} a_{46} + \frac{11117}{1458} a_{47} + \frac{11117}{1458} a_{48} + \frac{11117}{1458} a_{49}\\
+ \frac{8797}{1458} a_{50} + \frac{8797}{1458} a_{51} + \frac{8797}{1458} a_{52} + \frac{8797}{1458} a_{53} + \frac{2680}{729} a_{54} + \frac{2680}{729} a_{55} - \frac{2441}{729} a_{56} - \frac{2441}{729} a_{57} - \frac{2441}{729} a_{58} - \frac{2441}{729} a_{59}\\
- \frac{3359}{1458} a_{60} - \frac{3359}{1458} a_{61} - \frac{2239}{1458} a_{62} - \frac{2239}{1458} a_{63} - \frac{151}{162} a_{64} - \frac{2479}{1458} a_{65} - \frac{2479}{1458} a_{66} - \frac{151}{162} a_{67} + \frac{1298}{81} a_{68} + \frac{1298}{81} a_{69}\\
- \frac{659}{243} a_{70} - \frac{719}{243} a_{71} - \frac{719}{243} a_{72} - \frac{659}{243} a_{73} - \frac{611}{729} a_{74} + \frac{79}{729} a_{75} - \frac{611}{729} a_{76} - \frac{1301}{729} a_{77} - \frac{2933}{1458} a_{78} - \frac{4733}{1458} a_{79} - \frac{4733}{1458} a_{80}\\
- \frac{2933}{1458} a_{81} + \frac{1745}{729} a_{82} + \frac{2185}{729} a_{83} + \frac{1745}{729} a_{84} + \frac{145}{81} a_{85} - \frac{5333}{2916} a_{86} - \frac{5333}{2916} a_{87} - \frac{319}{108} a_{88} - \frac{319}{108} a_{89} + \frac{604}{729} a_{90}\\
+ \frac{464}{729} a_{91} + \frac{874}{729} a_{92} + \frac{338}{243} a_{93} + \frac{764}{729} a_{94} + \frac{118}{243} a_{95} + \frac{494}{729} a_{96} + \frac{904}{729} a_{97} - \frac{523}{1458} a_{98} - \frac{3433}{2916} a_{99} - \frac{7}{27} a_{100} - \frac{7}{27} a_{101}\\
+ \frac{248}{729} a_{102} + \frac{248}{729} a_{103} + \frac{248}{729} a_{104} + \frac{248}{729} a_{105} - \frac{1}{6} a_{106} - \frac{1}{6} a_{107} - \frac{103}{1458} a_{108} - \frac{103}{1458} a_{109} - \frac{71}{1458} a_{110} - \frac{71}{1458} a_{111}\\
+ \frac{10}{243} a_{112} + \frac{10}{729} a_{113} - \frac{10}{729} a_{114} + \frac{10}{729} a_{115} - \frac{85}{2916} a_{116} - \frac{85}{2916} a_{117} + \frac{10}{729} a_{119} + \frac{10}{729} a_{120} + \frac{10}{729} a_{121} + \frac{10}{729} a_{122}\\
- \frac{10}{729} a_{135} - \frac{10}{729} a_{136} + \frac{1}{2916} a_{195}
\end{multline*}
represents the flip.

For $X = \rO^-(6, 3)$, the elements $a_i$ up to $i = 221$ span a subspace of dimension $134$.
The flip is represented by a linear combination of the $a_i$ with coefficients as in Table~\ref{table:coeff-flip}.
\end{proof}

\section{Quantum graphs}\label{sec:quant-graphs}

Our last goal is to understand the quantum groups $\Aut^+(X)$ for $X = \VO^\epsilon(2 k, 2)$.
We are going to show that they are monoidally equivalent to $\PO(2^k)$ or $\PSp(2^k)$.
We do this by showing that $X$ is quantum isomorphic to quantum graphs on $M_{2^k}(\bC)$ on which these compact groups act as the quantum automorphism groups.

\subsection{2-categorical approach}

A `concrete' \emph{directed quantum graph} $X$ is given by a finite dimensional C$^*$-algebra endowed with a faithful normal state, $(B, \phi)$, (Frobenius C$^*$-algebra), and a linear map $A_X \colon B \to B$ such that
\begin{equation}\label{eq:idempot-Schur-prod}
m_B (A_X \otimes A_X) m_B^* = A_X,
\end{equation}
see~\cite{MR3849575}.
This structure can also be represented by a $B$-bimodule map $P_X\colon B \otimes B \to B \otimes B$ satisfying $P_X^2 = P_X$~\cite{MR3849575}*{Theorem 7.7}.

Let $(\cC_{i j})_{i, j = 0, 1}$ be a rigid C$^*$-$2$-category with two $0$-cells, $0$ and $1$.
Take $M \in \cC_{0 1}$, and consider the C$^*$-Frobenius algebra object $Q = M \otimes \bar M$ in $\cC = \cC_{0 0}$ whose structures are induced by a standard solution $(R_M, \bar R_M)$ for $M$.
We define a \emph{$2$-categorical directed quantum graph} $X$ on $Q$ to be $(Q, A_X)$, where $A_X \in \End_{\cC}(Q)$ satisfies the condition~\eqref{eq:idempot-Schur-prod} with $Q$ in place of $B$.
Again this structure can be equivalently given by an idempotent $P_X \in \End_{\hat\cC}(\hat Q)$ (`edge projection'), where $\hat \cC = \cC_{1 1}$ and $\hat Q = \bar M \otimes M$.

To obtain a concrete quantum graph from a $2$-categorical one, we need the following:
\begin{itemize}
\item a fiber functor $F \colon \cC \to \Hilbf$; and
\item a functor of left $\cC$-module categories from $\cC_{0 1} \simeq \rmodc{Q}$ to $\rmodc{B}$, again denoted by $F$, where $B = F(Q)$.
\end{itemize}
Since the $2$-categorical structure of $(\cC_{i j})_{i, j}$ is completely determined by $\cC$ and $Q$, these give rise to C$^*$-functors
\begin{align*}
\hat \cC &\to \bmodc{B},&
\cC_{1 0} \to \lmodc{B}
\end{align*}
(all denoted by $F$), which together form an extension to a functor of $2$-categories.
Then $B$ becomes a Frobenius C$^*$-algebra with the functional $\phi = F(\bar R_M^*)$, and we obtain a quantum graph $F(X)$ on $(B, \phi)$ given by $A_{F(X)} = F(A_X)$, such that $P_{F(X)} = F(P_X)$.

Let $G$ be the compact quantum group by the Tannaka--Krein duality from $(\cC, F)$.
By construction, $G$ is a quantum subgroup of the quantum automorphism group $\Aut^+(X)$.

\subsection{Application to affine orthogonal polar graphs}

Let $V = \bC^{2^k}$, and put
\[
B_k = \End_{\bC}(V) \simeq M_{2^k}(\bC)
\]
endowed with the tracial state $\phi = \tr_{2^k}$.
Consider the bilinear form $\Phi\colon V \times V \to \bC$ characterized by $\Phi(e_i, e_j) = \delta_{i, j}$ for the standard basis $(e_i)_{0 \le i < 2^k}$, and use this to identify $B_k$ with $V \otimes V$.
Thus, the matrix unit $e_{i j}$ corresponds to $e_i \otimes e_j$.

Consider the flip map on $V \otimes V$.
(This corresponds to the transpose map of $B_k$.)
Its eigendecomposition gives
\[
V \otimes V \simeq \Sym^2(V) \oplus \medwedge^2(V).
\]
Let $S_k = S^s_k$ be the orthogonal complement of the invariant vector $v^s_k = \sum_i e_i \otimes e_i$ (identified with $1_A$) in $\Sym^2(V)$, and let $P^s_k$ denote the orthogonal projection of $V \otimes V$ to $S_k$.

Consider the rigid $2$-category
\begin{align*}
\cC^{\rO}_{0 0} &= \cC^{\rO}_{1 1} = \Rep \PO(2^k),&
\cC^{\rO}_{0 1} &= \cC^{\rO}_{1 0} = (\Rep \rO(2^k))_1,
\end{align*}
where $(\Rep \rO(2^k))_1$ is the odd part of $\Rep \rO(2^k)$: the subcategory of representations on which the central element $-I_{2^k} \in \rO(2^k)$ acts by the natural character.
Then $V$ represents an object in $\cC^{\rO}_{0 1}$, hence we obtain a $2$-categorical quantum graph $X^s_k$ on $Q = V \otimes V$ whose edge projection is $P^s_k$.
Let $A^s_k$ denote the corresponding adjacency operator.

The canonical fiber functor $F_{\PO} \colon \cC^{\rO}_{0 0} \to \Hilbf$ admits a compatible module category functor
\[
\cC^{\rO}_{0 1} \to \rmodc{B_k}, \quad W \mapsto W \otimes V.
\]
Using these we obtain a quantum graph $F_{\PO}(X^s_k)$ on $(B_k, \phi)$.

We are going to show the following.

\begin{theorem}\label{thm:aff-pol-plus-q-iso-orth}
The graph $\VO^+(2 k, 2)$ is quantum isomorphic to $F_{\PO}(X^s_k)$.
\end{theorem}

\begin{theorem}\label{thm:pso-iso-q-aut}
We have $\PO(2^k) = \Aut^+(F_{\PO}(X^s_k))$.
\end{theorem}

\begin{corollary}\label{cor:VO-plus-PO-2-k}
The compact quantum groups $\Aut^+(\VO^+(2 k, 2))$ and $\PO(2^k)$ are monoidally equivalent.
\end{corollary}

Let us start with a concrete description of $\VO^+(2 k, 2)$ following~\cite{MR4350112}.
Let $\bF_2 = \{0, 1\}$ be the field with $2$ elements.
The vertex of $\VO^+(2 k, 2)$ set is $\rV(2 k, 2) = \bF_2^{2 k}$, and two vertices
\[
(x_0, y_0, \dots, x_{2 k - 1}, y_{2 k - 1}), 
(x'_0, y'_0, \dots, x'_{2 k - 1}, y'_{2 k - 1}) \quad (x_i, y_i, x'_i, y'_i \in \bF_2)
\]
are connected by an edge if and only if they are different and $\sum_i (x_i - x'_i) (y_i - y'_i) = 0$.
Thus, it is the Cayley graph of the additive group of $\rV(2 k, 2)$ with respect to the subset
\[
Y_k = Y^h_k = \{v = (x_0, y_0, \dots, x_{2 k - 1}, y_{2 k - 1}) \mid v \neq 0, Q_h(v) = 0 \},
\]
where $Q_h$ is the quadratic form on $\rV(2 k, 2)$ given by
\[
Q_h(x_0, y_0, \dots, x_{2 k - 1}, y_{2 k - 1}) = \sum_i x_i y_i.
\]

Again this fits in a $2$-categorical framework, as follows.
Take the commutative group $\Gamma = \Gamma_k = \rV(2 k, 2)$, and let $\hat\Gamma$ be its Pontryagin dual.
(These are of course isomorphic, but we will keep them separate for now.)
Then we have a $2$-category with components
\begin{align*}
\cC_{0 0} &= \lmodc{\Gamma},&
\cC_{0 1} &= \cC_{1 0} = \Hilbf,&
\cC_{1 1} &= \lmodc{\hat\Gamma},
\end{align*}
and the actions $\cC_{0 0} \curvearrowright \cC_{0 1} \curvearrowleft \cC_{1 1}$ are given by the canonical fiber functors.
Let $M \in \cC_{0 1}$ be the object represented by $\bC$.
Then $Q = M \otimes \bar M$ is the algebra $D = C(\Gamma)$ in $\cC = \cC_{0 0}$ with the translation action of $\Gamma$, while $\hat Q = \bar M \otimes M$ is $C(\hat\Gamma)$, again with the translation action of $\hat\Gamma$.

Note that $\hat\cC = \cC_{1 1}$ can be interpreted as the category of $\Gamma$-graded finite dimensional Hilbert spaces.
Then $\hat Q$ is identified with $C^*(\Gamma)$ with the natural grading.
From this we obtain the $2$-categorical quantum graph $X'$ on $Q$ whose edge projection $P'$ is the projection onto the span of $Y_k$ in $\hat Q$.

Let $H$ denote $C^*(\Gamma) = C(\hat\Gamma)$ regarded as a Hopf $*$-algebra, and similarly let $\hat H$ be $C^*(\hat\Gamma) = C(\Gamma)$.
Then $\cC_{1 1} = \rmodc{\hat H}$ is monoidally equivalent to $\Gamma\mhyph\bmodc{D}$ (with the bimodule tensor product) by the functor
\[
\Gamma\mhyph\bmodc{D} \to \rmodc{\hat H}, \quad E \mapsto \delta_{0} E,
\]
where the right $\hat H$-module structure is given by the algebra identification $\hat H = D$.
An inverse functor is given by
\[
\rmodc{\hat H} \to \Gamma\mhyph\bmodc{D}, \quad W \mapsto \oplus_{\gamma \in \Gamma} W.
\]
The left action of $D$ simply corresponds to the above labeling of direct summands, while the right action is the twisting of original right action by the translation by $g$ on the $g$-th summand.
Then we have an equivalent description of the above $2$-category, given by
\begin{align*}
\cC'_{0 0} &= \cC,&
\cC'_{0 1} &= \Gamma\mhyph\rmodc{D},&
\cC'_{1 0} &= \Gamma\mhyph\lmodc{D},&
\cC'_{1 1} &= \Gamma\mhyph\bmodc{D}.
\end{align*}

Now, suppose that $F\colon \cC \to \Hilbf$ is a fiber functor (which corresponds to a $\bT$-valued $2$-cocycle on $\hat\Gamma$).
Combined with the above model $\cC'$, this extends to a functor of $2$-categories such that
\begin{align*}
\cC_{0 1} &\to \rmodc{F(D)},&
\hat\cC &\to \bmodc{F(D)},
\end{align*}
and we obtain a quantum graph on $F(D)$.
When $F$ is the canonical fiber functor of $\cC = \lmodc\Gamma$, the projection $P'' = F(P')$ on $F(\hat Q) = C(\Gamma \times \Gamma)$ recovers the edge projection of $\VO^+(2 k, 2)$.

\begin{proof}[Proof of Theorem~\ref{thm:aff-pol-plus-q-iso-orth}]
Let $\omega = \omega_k$ be the $2$-cocycle on $\hat\Gamma$ given by
\[
\omega_k((\hat x_0, \hat y_0, \dots, \hat x_{2 k - 1}, \hat y_{2 k - 1}), (\hat x'_0, \hat y'_0, \dots, \hat x'_{2 k - 1}, \hat y'_{2 k - 1})) = \exp\biggl(\frac{\pi i}4 \sum_i \hat x_i \hat y'_i - \hat y_i \hat x'_i\biggr).
\]
Let $F_{\omega}$ be the associated fiber functor.
Then the twisted group algebra $C^*_\omega(\hat\Gamma) = F_{\omega}(D)$ is isomorphic to $B_k = M_{2^k}(\bC)$.
We claim that $F_{\PO}(X^s_k)$ is isomorphic to the quantum graph $X'_k$ on $(B_k, \phi)$ whose edge projection is $P'' = F_\omega(P')$.

First consider the case $k = 1$.
Then quantum graph on $X'_1$ on $M_2(\bC)$ is a simple graph with two quantum edges in the sense of~\cite{arXiv:2109.13618}.
The same holds for $F_{\PO}(X^s_1)$, hence these are isomorphic by~\citelist{\cite{arXiv:2109.13618}*{Section 3}\cite{arXiv:2110.09085}}.

The general case follows from this and the fact that the graphs follow the same recursion relation.
To be more precise, observe
\[
Y_{k+1} = (Y_k \times Y_1) \cup (Y_k \times \{0\}) \cup (\{0\} \times Y_1) \cup (Z_k \times Z_1),
\]
with
\[
Z_k = Z^h_k = \{v = (x_0, y_0, \dots, x_{2 k - 1}, y_{2 k - 1}) \mid Q_h(v) = 1 \}.
\]

On one hand, tensor product of edge projections correspond to tensor product of adjacency operators.
On the other, consider the $2$-categorical quantum graph on $Q = D \in \cC$ whose edge projection is the orthogonal projection to the span the unit element in $\hat Q = C^*(\Gamma)$ (the standard trace).
The adjacency operator of this graph is the identity map on $Q$.

Combining these two, the adjacency operators $A'_k = A_{X'_k}$ on $C^*_\omega(\Gamma)$ satisfy
\begin{equation}\label{eq:vo-adj-recurs}
A'_{k+1} = A'_k \otimes A'_1 + A'_k \otimes \id + \id \otimes A'_1 + A''_k \otimes A''_1,
\end{equation}
where $A''_k$ is the adjacency operator corresponding to the projection to the span of $Z_k$ in $C^*(\Gamma)$.
Since we have $\Gamma = \{0\} \coprod Y_k \coprod Z_k$, we have $\id + A'_k + A''_k$ is the adjacency operator for the quantum graph whose edge projection is identity.

Let us check the corresponding recursion relation for $X^s_k$.
The invariant vector $v^s_{k+1}$ can be identified with $v^s_k \otimes v^s_1$ up to rearranging tensor factors.
This implies
\[
S_{k+1} = S_k \otimes S_1 \oplus S_k \otimes \bC v^s_1 \oplus \bC v^s_{k} \otimes S_1 \oplus \medwedge^2(\bC^{2^k}) \otimes \medwedge^2(\bC^{2}).
\]

Consider the quantum graph on $V \otimes V \in \cC^{\PO}_{0 0}$ whose edge projection is the orthogonal projection to $\bC v^s_k \subset V \otimes V \in \cC^{\PO}_{1 1}$ ($R_V R^*_V$ for the duality structure morphism).
The adjacency operator is again identity morphism.

Thus, we have
\[
F_{\PO}(A^s_{k+1}) = F_{\PO}(A^s_k) \otimes F_{\PO}(A^s_1) + F_{\PO}(A^s_k) \otimes \id + \id \otimes F_{\PO}(A^s_1) + F_{\PO}(A^a_k) \otimes F_{\PO}(A^a_1),
\]
where $A^a_k$ is the adjacency operator of the $2$-categorical quantum graph whose edge projection is the orthogonal projection to $\medwedge^2(V) \subset V \otimes V \in \cC^{\PO}_{1 1}$.
Again we note that $\id + A^s_k + A^a_k$ is the adjacency operator for the quantum graph whose edge projection is identity.
From this we obtain the same recursion relation as~\eqref{eq:vo-adj-recurs} for the maps $A^s_k$.
\end{proof}

\begin{proof}[Proof of Theorem~\ref{thm:pso-iso-q-aut}]
We already have a monoidal functor $\Rep \Aut^+(\VO^+(2 k, 2)) \to \Rep \PO(2^k)$ by Theorem~\ref{thm:aff-pol-plus-q-iso-orth}.
This functor hits the flip map of $V \otimes V$.
Thus, the fiber functor on $\Rep \Aut^+(\VO^+(2 k, 2))$ coming from the fiber functor of $\Rep \PO(2^k)$ defines a usual compact group $G$, such that $G = \Aut^+(B_k, T_k)$.
We then have $\PO(2^k)$ as a subgroup of $G$, and the equality $\PO(2^k) = G$ is an immediate consequence of the fact that $G$ preserves the form $\Phi$.
\end{proof}

The comparison between $\Aut^+(\VO^-(2 k, 2))$ and $\PSp(2^k)$ is similar.
This time, we consider the bilinear form $\Psi\colon V \times V \to \bC$ characterized by
\begin{align*}
\Psi(e_{2 i}, e_{2 i + 1}) &= 1,&
\Psi(e_{2 i + 1}, e_{2 i}) &= -1,&
\Psi(e_i, e_j) & = 0 \quad (\absv{i - j} \neq 1)
\end{align*}
for the standard basis $(e_i)_{0 \le i < 2^k}$, and use this to identify $B_k$ with $V \otimes V$.

Let $S_k = S^a_k$ be the orthogonal complement of $v^a_k = \sum_i e_{2i} \otimes e_{2i+1} - e_{2i+1} \otimes e_{2i}$ (identified with $1_{B_k}$) in $\medwedge^2(V)$, and let $P^a_k$ denote the orthogonal projection of $V \otimes V$ to $S_k$.

As before, consider the rigid $2$-category
\begin{align*}
\cC^{\PSp}_{0 0} &= \cC^{\PSp}_{1 1} = \Rep \PSp(2^k),&
\cC^{\PSp}_{0 1} &= \cC^{\PSp}_{1 0} = (\Rep \PSp(2^k))_1.
\end{align*}
Again $V$ represents an object in $\cC^{\PSp}_{0 1}$, and we obtain a quantum graph $X^a_k$ on $Q = V \otimes V \in \cC^{\PSp}_{0 0}$ whose edge projection is $P^a_k$.
We denote the corresponding adjacency operator by $A^a_k$.
Moreover, the canonical fiber functor $F_{\PSp} \colon \cC^{\PSp}_{0 0} \to \Hilbf$ extends to a module category functor
\[
F_{\PSp} \colon \cC^{\PSp}_{0 1} \to \rmodc{B_k}, \quad W \mapsto W \otimes V.
\]
Thus we obtain a quantum graph $F_{\PSp}(X^a_k)$ on $B_k$.
We then have the following analogues of Theorems~\ref{thm:aff-pol-plus-q-iso-orth},~\ref{thm:pso-iso-q-aut}, and Corollary~\ref{cor:VO-plus-PO-2-k}.

\begin{theorem}\label{thm:aff-pol-minus-q-iso-symp}
The graph $\VO^-(2 k, 2)$ is quantum isomorphic to $F_{\PSp}(X^a_k)$.
\end{theorem}

\begin{theorem}\label{thm:psp-iso-q-aut}
We have $\PSp(2^k) = \Aut^+(F_{\PSp}(X^a_k))$.
\end{theorem}

\begin{corollary}\label{cor:VO-minus-PSp-2-k}
The compact quantum groups $\Aut^+(\VO^-(2 k, 2))$ and $\PSp(2^k)$ are monoidally equivalent.
\end{corollary}

\begin{proof}[Proof sketch]
This time we consider the quadratic form
\[
Q_e(x_0, y_0, \dots, x_{2 k - 1}, y_{2 k - 1}) = Q_h(x_0, y_0, \dots, x_{2 k - 2}, y_{2 k - 2}) + x_{2 k - 1}^2 + y_{2 k - 1}^2 + x_{2 k - 1} y_{2 k - 1}
\]
that gives $\VO^-(2 k, 2)$ as the Cayley graph with respect to the subset
\[
Y_k = Y^e_k = \{v = (x_0, y_0, \dots, x_{2 k - 1}, y_{2 k - 1}) \mid v \neq 0, Q_e(v) = 0 \}.
\]
Since $Y^e_1 = \emptyset$, the recursive relation is
\[
Y^e_{k+1} = Y^h_k \times \{0\} \cup Z^h_k \times Z^e_1
\]
for
\[
Z^e_1 = \rV(2, 2) \setminus \{ 0 \} = \{v = (x_0, y_0) \mid Q_e(v) = 1 \}.
\]

As for the vectors $v^a_k$ and the spaces $S^a_k$, we have
\begin{align*}
v^a_{k + 1} &= v^s_k \otimes v^a_1,&
S^a_{k + 1} &= S^s_k \otimes \bC v^a_1 \oplus \medwedge^2(\bC^{2^k}) \otimes \Sym^2(\bC^2)
\end{align*}

Since we know the correspondence between $Y^h_k$ and $S^s_k$, it only remains to compare the quantum graphs on $M_2(\bC)$ whose edge projections correspond to $Z^e_1$ and $\Sym^2(\bC^2)$.
But both of them gives the simple quantum graph with three edges, hence we are done.
\end{proof}

\appendix
\section{Computer algebra computation}\label{sec:appendix}

The computation for $X = \rO^-(6, q)$ for $q = 2, 3$ and the McLaughlin graph was implemented as \texttt{python} programs based on the \texttt{SageMath} library.
The grogram files in \texttt{Jupyter Notebook} format and their PDF renderings are attached as supplementary files.

In these programs, we compute the coefficients $c_\Gamma^{x_0, \dots, x_3} d_\Gamma^{x_0, \dots, x_3}$ in~\eqref{eq:a-Gamma-element}, which is the number of compatible assignments of vertices of $X$ to vertexes of $\Gamma$, by an optimized search algorithm.

The graph $\Gamma$ is specified by the following parameters:
\begin{itemize}
\item an integer representing the number of external vertices, which is equal to the integer $k$ such that $a_\Gamma \in P^X_{k,+}$;
\item a partition of the set $\{0, \dots, k - 1 \}$, representing the grouping of vertices joined by double strike edges;
\item a list of pairs of integers $(i, j)$, representing single strike edges between the $i$-th and the $j$-th external vertices;
\item a list of pairs of integers $(i, j)$, representing single strike edges between the $i$-th external and the $j$-th internal vertices;
\item a list of pairs of integers $(i, j)$, representing single strike edges between the $i$-th and the $j$-th internal vertices; and
\item an integer $n$, representing the number of internal vertices.
\end{itemize}
The main counting algorithm, implemented inside the \texttt{coeff\_func} function, takes candidates $x_0, \dots, x_{k-1}$ for the external vertices as inputs, and returns the number of $n$-tuples of vertices of $X$ fitting the above conditions for the internal vertices of $\Gamma$.
It first decides if the $x_i$ are consistent with the partition, and returns $0$ if not.
It then starts counting over the $n$-tuples of vertices $y_0, \ldots, y_{n-1}$ matching the conditions specified by the three kinds of lists of pairs as above.
We use the \texttt{SageMath} library function to get list of neighbors for vertices in $X$, and discard the incompatible candidates for the $y_j$ as early as possible.

In case of the McLaughlin graph, and the $\Gamma_j$ as in Table~\ref{tab:graphs}, the program completes in about 100 hours divided by the number of CPUs running at 3.2 GHz.

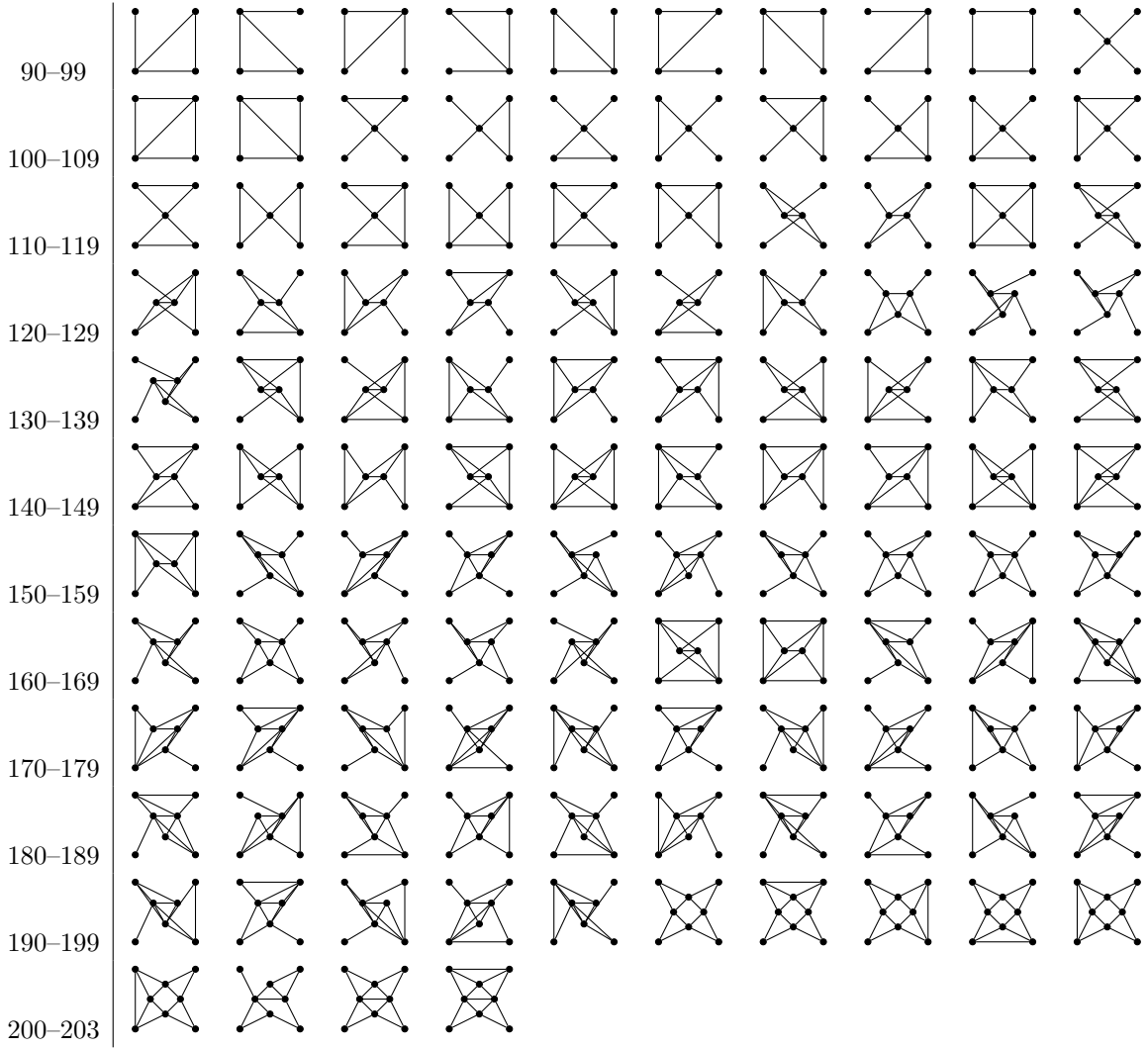
\begin{longtable}[l]{c|cccccccccc}
\caption{graphs with four external vertices and up to nine edges, and one extra with twelve.
The external vertex $0$ is at the top left corner for each graph.}\label{tab:graphs}\\
$j$ & $\Gamma_j$ \\
\end{longtable}

\subsection{Extra elements for \texorpdfstring{$\rO^-(6,3)$}{O-(6,3)}}

We focus on the elements of $P^{\rO^-(6,3)}_{4,+}$.
Suppose we have $a, b \in P^{\rO^-(6,3)}_{4,+}$ presented as in~\eqref{eq:elem-P-k-from-coeff}.
We then have their product $a * b$ with the coefficient
\begin{equation}\label{eq:prod-in-P-4}
(a * b)^{x_0, x_1, x_2, x_3} = \sum_{y, z} a^{x_0, x_1, y, z} b^{z, y, x_2, x_3},
\end{equation}
together with the rotation $\operatorname{rot}(a)$ and inverse rotation $\operatorname{invrot}(a)$ given by
\[
\operatorname{rot}(a)^{x_0, x_1, x_2, x_3} = a^{x_3, x_0, x_1, x_2}, \quad
\operatorname{invrot}(a)^{x_0, x_1, x_2, x_3} = a^{x_1, x_2, x_3, x_0},
\]
again in $P^{\rO^-(6,3)}_{2,+}$.

To compute the product~\eqref{eq:prod-in-P-4}, we exploit the fact that $a^{x_0, \cdots, x_3}$ and $b^{x_0, \cdots, x_3}$ are constant on each $\Aut(X)$-orbit of $V^4$.
Thus, to determine the value of~\eqref{eq:prod-in-P-4} for a given quadruple $x_0, \dots, x_3$, we can count the orbits $\omega_1, \omega_2 \subset V^4$ respectively representing $(x_0, x_1, e, f)$ and $(f, e, x_2, x_3)$ for all possibple choices of $e$ and $f$, and keep the multiplicities of $(\omega_1, \omega_2)$ in this counting (which is independent of the choice of $a \in P^X_{4,+}$) implemented in the \texttt{prod\_mult} function.
For $X = O^-(6, 3)$, this counting for all orbit types for $x_0, \dots, x_3$ takes about 130 hours divided by the number of CPUs running at 3.2 GHz.
To facilitate the computation, once we generate this multiplicity table, we store it in a data file (\texttt{tup.pickle}) and load its content for later calculations.

Using these operations, we find additional elements as follows, giving the $134$-dimensional intertwiner space on $C(V \times V)$ for the quantum automorphism group, as expected from the classical automorphism group.
\begin{itemize}
\item Element 204: product of 195-th and 195-th element.
\item Element 205: product of 195-th and 198-th element.
\item Element 206: rotation of 204-th element.
\item Element 207: rotation of 205-th element.
\item Element 208: product of 195-th and 206-th element.
\item Element 209: product of 195-th and 207-th element.
\item Element 210: product of 198-th and 207-th element.
\item Element 211: product of 200-th and 207-th element.
\item Element 212: product of 201-th and 207-th element.
\item Element 213: product of 207-th and 205-th element.
\item Element 214: product of 207-th and 207-th element.
\item Element 215: rotation of 209-th element.
\item Element 216: rotation of 210-th element.
\item Element 217: inverse rotation of 212-th element.
\item Element 218: product of 65-th and 217-th element.
\item Element 219: product of 67-th and 216-th element.
\item Element 220: product of 213-th and 213-th element.
\item Element 221: product of 213-th and 214-th element.
\end{itemize}
For example, Element 209 is represented by the graph of Figure~\ref{fig:elem-209}.

\begin{figure}[h]

\def\globalscale{1.000000}
\begin{tikzpicture}[y=1pt, x=1pt, yscale=\globalscale,xscale=\globalscale, every node/.append style={scale=\globalscale}, inner sep=0pt, outer sep=0pt]
  \path[draw=black,line cap=,line width=0.5pt] (94.966, 728.891) -- (83.716, 
  717.641) -- (94.966, 706.391) -- (106.216, 717.641) -- cycle;

  \path[draw=black,line cap=,line width=0.5pt] (83.716, 717.641) -- (76.216, 
  698.891) -- (94.966, 706.391) -- (113.716, 698.891) -- (106.216, 717.641) -- 
  (113.716, 736.391) -- (94.966, 728.891) -- (76.216, 736.391) -- cycle;

  \path[draw=black,line cap=,line width=0.5pt] (132.466, 728.891) -- (121.216, 
  717.641) -- (132.466, 706.391) -- (143.716, 717.641) -- cycle;

  \path[draw=black,line cap=,line width=0.5pt] (121.216, 717.641) -- (113.716, 
  698.891) -- (132.466, 706.391) -- (151.216, 698.891) -- (143.716, 717.641) -- 
  (151.216, 736.391) -- (132.466, 728.891) -- (113.716, 736.391) -- cycle;

  \path[draw=black,line cap=,line width=0.5pt] (113.716, 772.391) -- (102.466, 
  761.141) -- (113.716, 749.891) -- (124.966, 761.141) -- cycle;

  \path[draw=black,line cap=,line width=0.5pt] (102.466, 761.141) -- (76.216, 
  736.391) -- (113.716, 749.891) -- (151.216, 736.391) -- (124.966, 761.141) -- 
  (132.466, 779.891) -- (113.716, 772.391) -- (94.966, 779.891) -- cycle;

  \path[draw=black,line cap=,line width=0.5pt] (151.216, 698.891) -- (151.716, 
  736.391);

  \path[draw=black,fill=black,line cap=,line width=0.25pt] (102.468, 761.519) 
  circle (0.937pt);

  \path[draw=black,fill=black,line cap=,line width=0.25pt] (124.968, 761.519) 
  circle (0.937pt);

  \path[draw=black,fill=black,line cap=,line width=0.25pt] (113.718, 772.769) 
  circle (0.937pt);

  \path[draw=black,fill=black,line cap=,line width=0.25pt] (113.718, 750.269) 
  circle (0.937pt);

  \path[draw=black,fill=black,line cap=,line width=0.25pt] (83.718, 717.269) 
  circle (0.937pt);

  \path[draw=black,fill=black,line cap=,line width=0.25pt] (106.218, 717.269) 
  circle (0.937pt);

  \path[draw=black,fill=black,line cap=,line width=0.25pt] (94.968, 728.519) 
  circle (0.937pt);

  \path[draw=black,fill=black,line cap=,line width=0.25pt] (94.968, 706.019) 
  circle (0.937pt);

  \path[draw=black,fill=black,line cap=,line width=0.25pt] (121.218, 717.269) 
  circle (0.937pt);

  \path[draw=black,fill=black,line cap=,line width=0.25pt] (143.718, 717.269) 
  circle (0.937pt);

  \path[draw=black,fill=black,line cap=,line width=0.25pt] (132.468, 728.519) 
  circle (0.937pt);

  \path[draw=black,fill=black,line cap=,line width=0.25pt] (132.468, 706.019) 
  circle (0.937pt);

  \path[draw=black,fill=black,line cap=,line width=0.25pt] (113.718, 736.019) 
  circle (0.937pt);

  \path[draw=black,fill=black,line cap=,line width=0.25pt] (113.718, 699.269) 
  circle (0.937pt);

  \path[draw=black,fill=black,line cap=,line width=0.25pt] (76.218, 736.019) 
  circle (0.937pt);

  \path[draw=black,fill=black,line cap=,line width=0.25pt] (151.218, 736.019) 
  circle (0.937pt);

  \path[draw=black,fill=black,line cap=,line width=0.25pt] (76.218, 699.269) 
  circle (0.937pt);

  \path[draw=black,fill=black,line cap=,line width=0.25pt] (151.218, 699.269) 
  circle (0.937pt);

  \path[draw=black,fill=black,line cap=,line width=0.25pt] (94.968, 780.269) 
  circle (0.937pt);

  \path[draw=black,fill=black,line cap=,line width=0.25pt] (132.468, 780.269) 
  circle (0.937pt);

  \node[text=black,line cap=,line width=0.5pt,anchor=south west] (text26) at 
  (89.716, 783.391){$v^e_0$};

  \node[text=black,line cap=,line width=0.5pt,anchor=south west] (text27) at 
  (129.466, 783.391){$v^e_1$};

  \node[text=black,line cap=,line width=0.5pt,anchor=south west] (text28) at 
  (148.966, 687.391){$v^e_2$};

  \node[text=black,line cap=,line width=0.5pt,anchor=south west] (text29) at 
  (72.466, 687.391){$v^e_3$};

\end{tikzpicture}
\caption{Element 209, $v^e_i$ denoting the external vertices}\label{fig:elem-209}
\end{figure}
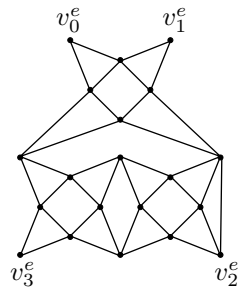

\twocolumn
\begingroup
\renewcommand\arraystretch{1.3}
\topcaption{Coefficients for the flip}\label{table:coeff-flip}
\begin{supertabular}{c|c}
$i$ & coefficient \\
\hline
0 & $- \frac{21025310690896414103783342291}{5264138895030288000000}$ \\
1 & $- \frac{775159348760990269703359247}{87735648250504800000}$ \\
2 & $\frac{3783980408877756001776041}{259957476297792000000}$ \\
3 & $- \frac{124822837096765253849051885573}{438678241252524000000}$ \\
4 & $\frac{6318600791316477269346343}{179970560513856000000}$ \\
5 & $\frac{5677343781809041942707629}{233961728668012800000}$ \\
6 & $\frac{5603735832572982531374861}{389936214446688000000}$ \\
7 & $\frac{608087230517011462976629079}{21056555580121152000000}$ \\
8 & $- \frac{1690603723396487842966226953}{21056555580121152000000}$ \\
9 & $- \frac{2453752911417442177423289467}{21056555580121152000000}$ \\
10 & $\frac{396860824441889550160649099}{3008079368588736000000}$ \\
11 & $- \frac{28951758557641654774576699}{8225217023484825000}$ \\
12 & $- \frac{1208618036590724771727470301019}{752019842147184000000}$ \\
13 & $- \frac{2050005815209742492885835337}{5264138895030288000000}$ \\
14 & $- \frac{6074004870753698602679941397}{1316034723757572000000}$ \\
15 & $- \frac{17293567340745128736343483}{2339617286680128000000}$ \\
16 & $- \frac{9967921877334091734731059}{802154498290329600}$ \\
17 & $- \frac{22381250680490198221750763}{2339617286680128000000}$ \\
18 & $\frac{37660957035372776051657}{51420160146816000000}$ \\
19 & $\frac{140209800516573882863879}{1559744857786752000000}$ \\
20 & $\frac{118395187592597701932106297}{7018851860040384000000}$ \\
21 & $\frac{49934704134745083104792569}{3239470089249408000000}$ \\
22 & $- \frac{40907453552330906702889577}{21056555580121152000000}$ \\
23 & $- \frac{48434353432765879475543347}{14037703720080768000000}$ \\
24 & $\frac{100493175082130667319971331}{7018851860040384000000}$ \\
25 & $\frac{14421958661249551497312109}{4679234573360256000000}$ \\
26 & $- \frac{499921631693359438748793583}{21056555580121152000000}$ \\
27 & $- \frac{591723381808583728099825673}{42113111160242304000000}$ \\
28 & $- \frac{34331249030109399810581475463}{404933761156176000000}$ \\
29 & $\frac{311373439594372577421588059}{42113111160242304000000}$ \\
30 & $- \frac{283292495474950590196242290209}{526413889503028800000}$ \\
31 & $\frac{820966527709408388603269769}{42113111160242304000000}$ \\
32 & $- \frac{24814732612005644419075931}{1619735044624704000000}$ \\
33 & $- \frac{76506219520402768685213461}{4211311116024230400000}$ \\
34 & $- \frac{504523592175457796076280063}{21056555580121152000000}$ \\
35 & $- \frac{439138366247692176257771977}{21056555580121152000000}$ \\
36 & $- \frac{85004076835501501130303057}{100269312286291200000}$ \\
37 & $- \frac{4523480799755027996506921381}{3509425930020192000000}$ \\
38 & $\frac{1704649393075018353199331743}{10528277790060576000000}$ \\
39 & $\frac{31407151288514058797090503}{1504039684294368000000}$ \\
40 & $\frac{181372036256408270433052153}{97484053611672000000}$ \\
41 & $\frac{69426960623479586085361171}{438678241252524000000}$ \\
42 & $\frac{3632738612611356417519577}{501346561431456000000}$ \\
43 & $\frac{17756729642614719631237199}{7018851860040384000000}$ \\
44 & $\frac{7809308361379868435412287}{7018851860040384000000}$ \\
45 & $\frac{4344858937599602432784313}{501346561431456000000}$ \\
46 & $\frac{378345907503906257466913}{119980373675904000000}$ \\
47 & $\frac{327938166650911785683377639}{42113111160242304000000}$ \\
48 & $\frac{2345240069187822446819947}{1203231747435494400000}$ \\
49 & $- \frac{2530253127740291388783929}{1203231747435494400000}$ \\
50 & $\frac{64213521069402390407030837}{4679234573360256000000}$ \\
51 & $- \frac{60076913716411309957137121}{42113111160242304000000}$ \\
52 & $\frac{91972419008734623631908431}{42113111160242304000000}$ \\
53 & $\frac{220940303454703192578438949}{42113111160242304000000}$ \\
54 & $- \frac{28131221176914158729598656957}{1052827779006057600000}$ \\
55 & $\frac{10759670886407811451175467}{1316034723757572000000}$ \\
56 & $- \frac{9670654433790878682261521}{194968107223344000000}$ \\
57 & $\frac{6568664772395890983688919}{1504039684294368000000}$ \\
58 & $\frac{4089918887005753650640307}{57847680165168000000}$ \\
59 & $- \frac{391304893103345694802227919}{10528277790060576000000}$ \\
60 & $- \frac{60228723504988608292879381}{779872428893376000000}$ \\
61 & $- \frac{46517804816269114630542257}{2339617286680128000000}$ \\
62 & $\frac{1792501608307179471500893141}{4211311116024230400000}$ \\
63 & $- \frac{2562643822465951898772959}{111410346984768000000}$ \\
64 & $\frac{99743337638949117510798139}{2339617286680128000000}$ \\
65 & $- \frac{77874728714889089517033007}{7018851860040384000000}$ \\
66 & $\frac{12327363318953623188834638387}{21056555580121152000000}$ \\
67 & $- \frac{16051108745785519948048603}{467923457336025600000}$ \\
68 & $- \frac{134354013527742926867854561}{2632069447515144000000}$ \\
69 & $\frac{28179686564533873357501501}{5264138895030288000000}$ \\
70 & $\frac{190468473324836374827033769}{10528277790060576000000}$ \\
71 & $\frac{63472551110146869334658983}{10528277790060576000000}$ \\
72 & $\frac{1755414013383050023926523}{526413889503028800000}$ \\
73 & $\frac{26918225339018981322685849}{5264138895030288000000}$ \\
74 & $- \frac{1266256489741361416481119}{263206944751514400000}$ \\
75 & $- \frac{3602299587263288289207353}{3509425930020192000000}$ \\
76 & $- \frac{1766124882991816897360051}{188004960536796000000}$ \\
77 & $\frac{8329653291974128124522623}{501346561431456000000}$ \\
78 & $- \frac{349256178260586738613178261}{42113111160242304000000}$ \\
79 & $- \frac{576828186158173893191973527}{42113111160242304000000}$ \\
80 & $\frac{11836299263539972226462893}{4679234573360256000000}$ \\
81 & $\frac{812081888145888689089379}{2005386245725824000000}$ \\
82 & $\frac{15352502384413187652829633}{25308360072261000000}$ \\
83 & $\frac{1507410969773893300536195889}{10528277790060576000000}$ \\
84 & $- \frac{290248847007076564710823511}{438678241252524000000}$ \\
85 & $\frac{89346495458233875284448289}{501346561431456000000}$ \\
86 & $- \frac{289267752647770282692999131}{42113111160242304000000}$ \\
87 & $- \frac{20991491500412962231559329}{2005386245725824000000}$ \\
88 & $- \frac{458180909353855495122321401}{42113111160242304000000}$ \\
89 & $\frac{409708784791239410761622459}{14037703720080768000000}$ \\
90 & $\frac{748359190716401019875537}{179970560513856000000}$ \\
91 & $- \frac{7685374823519151442703}{802337889808000000}$ \\
92 & $\frac{258290347078253213973666271}{21056555580121152000000}$ \\
93 & $- \frac{4067240220760043990478323}{421131111602423040000}$ \\
94 & $\frac{233204053737763361569283}{467923457336025600000}$ \\
95 & $- \frac{602475955877564467491766433}{10528277790060576000000}$ \\
96 & $\frac{229094132941971430352365933}{21056555580121152000000}$ \\
97 & $\frac{14271865984752071672391613}{350942593002019200000}$ \\
98 & $\frac{326547442577039744976329}{62668320178932000000}$ \\
99 & $- \frac{9553216903901834656822391}{2632069447515144000000}$ \\
100 & $\frac{583208470666217579890828567}{10528277790060576000000}$ \\
101 & $\frac{582766378504743432268855081}{10528277790060576000000}$ \\
102 & $- \frac{24853144753972129044732689}{10528277790060576000000}$ \\
103 & $\frac{4106647997386881253198387}{2005386245725824000000}$ \\
104 & $- \frac{7267185985016541196372271}{3509425930020192000000}$ \\
105 & $\frac{30137311627326192217053127}{14037703720080768000000}$ \\
106 & $- \frac{167393986118802382470553}{66846208190860800000}$ \\
107 & $\frac{18471186351808224848797337}{21056555580121152000000}$ \\
108 & $\frac{2455002067106392663553777}{3008079368588736000000}$ \\
109 & $- \frac{27041765619415861475962007}{7018851860040384000000}$ \\
110 & $\frac{1124194618777486169211931}{1169808643340064000000}$ \\
111 & $\frac{2688974097439499562596233}{87735648250504800000}$ \\
112 & $\frac{24094120956067598957637821}{752019842147184000000}$ \\
116 & $- \frac{56497279530375275113711853}{42113111160242304000000}$ \\
117 & $- \frac{92756214793146688220535023}{42113111160242304000000}$ \\
195 & $\frac{1852346431848209045979973}{105282777900605760000}$ \\
203 & $\frac{1293058632732080338549}{58490432167003200000}$ \\
204 & $- \frac{63954424343082186417307}{350942593002019200000}$ \\
205 & $\frac{972518265407684337031}{263206944751514400000}$ \\
206 & $- \frac{4477026894095613901}{562408001605800000}$ \\
207 & $- \frac{433961676362372527981}{4211311116024230400000}$ \\
208 & $- \frac{2378638485027753923}{7498773354744000000}$ \\
209 & $- \frac{1664547861260593189}{1124816003211600000}$ \\
210 & $- \frac{2888771619615658333}{2812040008029000000}$ \\
211 & $- \frac{822654652034767}{15622444489050000}$ \\
212 & $\frac{1394067727155571}{857002669113600000}$ \\
213 & $- \frac{193724700491}{1785422227320000}$ \\
214 & $\frac{193644819637177}{89271111366000000}$ \\
215 & $\frac{85865920907167}{33476666762250000}$ \\
216 & $- \frac{6727532642923}{11902814848800000}$ \\
217 & $\frac{10966168743983}{29995093418976000}$ \\
218 & $- \frac{308927622553433}{53562666819600000}$ \\
219 & $\frac{11540045183}{73130894431027200000}$ \\
220 & $- \frac{291263071}{13712042705817600000}$ \\
\end{supertabular}
\endgroup
\onecolumn

\end{document}